\documentclass[reqno]{amsart}
\newcommand \datum {August 10, 2021}
\usepackage{amssymb,latexsym}
\usepackage{amsmath}
\usepackage{amscd}
\usepackage{graphicx}
\usepackage{color}
\usepackage[dvipsnames]{xcolor} %For user-defined colors}
\usepackage{multibib} %For more than 1 bibliographic section
\usepackage{enumerate}
\numberwithin{equation}{section}
\theoremstyle{plain}
 \newtheorem{theorem}{Theorem}[section]
 \newtheorem{lemma}[theorem]{Lemma}
 \newtheorem{proposition}[theorem]{Proposition}
 \newtheorem{corollary}[theorem]{Corollary}
\theoremstyle{definition}
 \newtheorem{definition}[theorem]{Definition}

 \newtheorem{example}[theorem]{Example}
 \newtheorem{observation}[theorem]{Observation}
 \newtheorem{remark}[theorem]{Remark}
 \newtheorem{case}{Case} 
 \newtheorem{claim}{Claim} 
  
\newenvironment{enumeratei}{\begin{enumerate}[\quad\upshape (i)]} {\end{enumerate}}
\newcommand \Lext  {\pLext{L}{e}{h}}
\newcommand \pLext [3] {{#1}({#2}_\ast\kern-5pt\mathord{\searrow}{#3})}
\newcommand \wk [1]  {\widehat{#1}}

\newcommand \plue [1] {#1^{\kern-1pt+\kern-1pt e}}
\newcommand \mine [1] {#1^{\kern-1pt -\kern-1pt e}}
\newcommand \alit [1] {#1^{\kern-1.5pt\triangledown}}
\newcommand \lift [1] {#1^{\kern-1pt\blacktriangle}}

\newcommand \ogiso {\overline \phi_{\textup{geo}}}
\newcommand \giso {\phi_{\textup{geo}}}
\newcommand \liso {\phi_{\textup{lat}}}
\newcommand \Lclass {{\mathbf L}^{\kern-2pt\textup c}}
\newcommand \Gclass {{\mathbf G}^{\textup c}}
\newcommand \At [1] {\textup{At}(#1)}
\newcommand \fprec {\mathrel{\prec_{\alg F}}}
\newcommand \fpreceq {\mathrel{\preceq_{\alg F}}}
\newcommand \gprec {\mathrel{\prec_{\alg G}}}

\newcommand \ogmap {\overline\mu}
\newcommand \gmap {\mu}

\newcommand \swedge[1]{\mathop{\wedge\kern-1pt{}_#1}}
\newcommand \svee[1]{\mathop{\vee\kern-2pt{}_#1}}
\newcommand \sbigwedge[1]{\bigwedge\kern-2pt{}_#1}
\newcommand \sbigvee[1]{\bigvee\kern-3pt{}_#1}

\newcommand \nothing [1] {}

\newcommand \defiff {\overset{\textup{def}}\iff}
\newcommand \nonparallel {\mathrel{\not{\kern-1.4pt{\mathord\parallel}} }}
\newcommand \wha [1] {\alg{#1}^\ast}
\newcommand \mlat {\textup{Lat}}
\newcommand \mgeom {\textup{Geom}}

\newcommand \Lat [1]{\textup{Lat}(#1)}
\newcommand \Geom [1] {\textup{Geom}(#1)}
\newcommand \mclf {\textup{cl}_{\alg F}}
\newcommand \clf[1] {\textup{cl}_{\alg F}(#1)}

\newcommand \dwn {\mathord{\downarrow}}
\newcommand \Pow[1]{\textup{Pow}(#1)}

\newcommand \width[1]{\textup{width}(#1)}
\newcommand \then {\Rightarrow}
\newcommand \tbf[1]  {\textbf{#1}} 
\newcommand \alg[1]  {\mathcal #1}

\newcommand \Jir [1] {\textup J(#1)}

\newcommand \set [1]{\{#1\}}

\renewcommand \phi{\varphi}

\newcommand \lcov [2] {#2_{\ast#1}}

\newcommand \restrict [2] {#1\rceil_{\kern -1pt #2}}

\newcommand \sideal [2] {\mathord{\downarrow}_{\kern-1pt#1}\kern 1pt #2}
\newcommand \sfilter [2] {\mathord{\uparrow}_{\kern-1pt#1}\kern 1pt#2}
\newcommand \sfolter [2] {\mathord{\Uparrow}_{\kern-1pt#1}\kern 1pt #2}

\newcommand \sodeal [2] {\mathord{\Downarrow}_{\kern-1pt#1}\kern 1pt #2}

\newcommand \length [1]   {\textup{length}(#1)}

\newcommand \injour[1] {}

\newcommand\red[1]{{\textcolor{red}{#1}}}

\begin{document}
\title[Lowering a join-irreducible element of a semimodular lattice]
{{Length-preserving extensions of semimodular lattices by lowering  join-irreducible elements}}

\author[G.\ Cz\'edli]{G\'abor Cz\'edli}

\email{czedli@math.u-szeged.hu}
\urladdr{http://www.math.u-szeged.hu/~czedli/}
\address{ Bolyai Institute, University of Szeged, Hungary}

\begin{abstract} 
We prove that if $e$ is a join-irreducible element of a semimodular lattice $L$ of finite length and $h<e$ in $L$ such that $e$ does not cover $h$, then $e$ can be ``lowered'' to a covering of $h$ by taking a length-preserving semimodular extension $K$ of $L$ but not changing the rest of join-irreducible elements. With the help of our ``lowering construction'', we prove a general theorem on length-preserving  semimodular extensions of semimodular lattices, which  implies some earlier results  proved by G.\ Gr\"atzer and E.\ W.\ Kiss (1986), M. Wild (1993), and G.\ Cz\'edli and E.\ T.\ Schmidt (2010) on extensions to geometric lattices, and even an unpublished result of E.\ T.\ Schmidt on higher dimensional rectangular lattices. Our method offers shorter proofs of these results than the original ones. To obtain the main tool used in the paper, we extend the  bijective correspondence between finite semimodular lattices and Faigle geometries to an analogous correspondence between semimodular lattices of finite lengths and a larger class of geometries.
\end{abstract}

\thanks{This research was supported by the National Research, Development and Innovation Fund of Hungary under funding scheme K 134851.}

%06C10 (1980-now) Semimodular lattices, geometric lattices
%    51D (1980-now) Geometric closure systems
%51D25 (1980-now) Lattices of subspaces and geometric closure systems 
%   51E (1980-now) Finite geometry and special incidence structures
%51E99 (1980-now) None of the above, but in this section

\dedicatory{Dedicated to the memory of my father, Gy\"orgy}


\date{\datum.\hfill{{General rule: \boxed{\tbf{check}} the author's website for updates}%
}}

\leftline{\hfill{\tiny{Extended version, 10 typos have been corrected}}}

\maketitle

\section{Introduction}\label{sect:visualintro}

\begin{figure}[h]
\centerline
{\includegraphics[width=\textwidth]{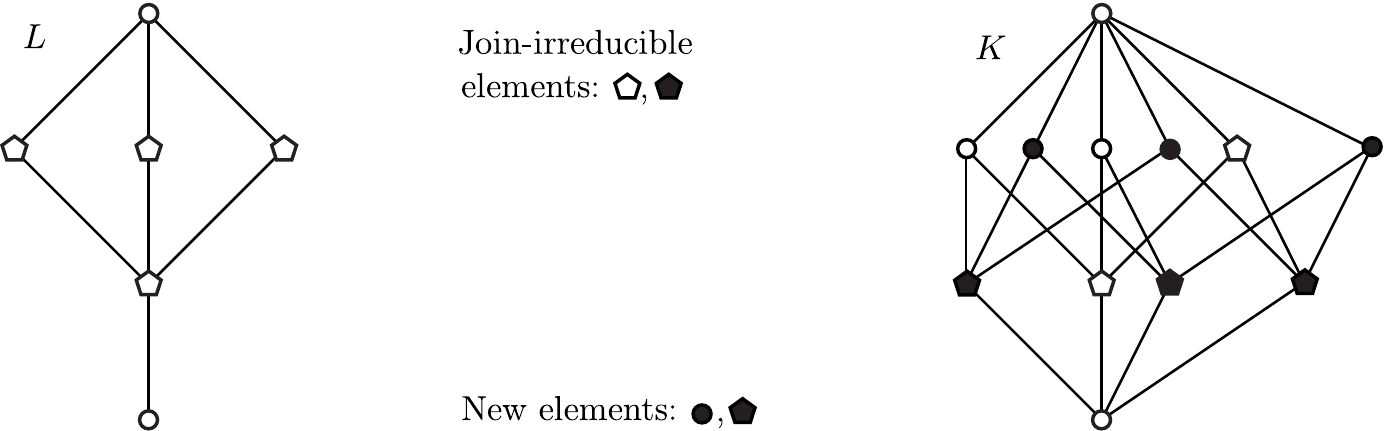}}      %[scale=0.93]{czgamfig1}}
\caption{A length-preserving embedding of $L$ into $K$}\label{figp3}
\end{figure}

There are known  constructions, to be listed later, that yield a cover-preserving $\set{0,1}$-embedding
of a finite semimodular lattice $L$ into a finite geometric lattice $K$; see  Figure~\ref{figp3} for an example. (We will point out later, right before the proof of Theorem~\ref{thmmain},  why $K$ is a geometric lattice.)
In this paper, we are going to extend $L$ to $K$ in \emph{elementary steps}. Each step only ``lowers'' one join-irreducible element; see Figures~\ref{figp4}--\ref{figp6}.  This gives better insight into the construction since our elementary steps are easier to understand than an immediate transition from $L$ to a geometric lattice.
As a consequence, our approach  leads to a general embeddability theorem. To provide an appropriate tool for our proofs, we introduce a \emph{class of geometries} that are in bijective correspondence with semimodular lattices of finite length. The finite ones of these geometries  are due to Faigle~\cite{faigle}.

\begin{figure}[h]
\centerline
{\includegraphics[width=\textwidth]{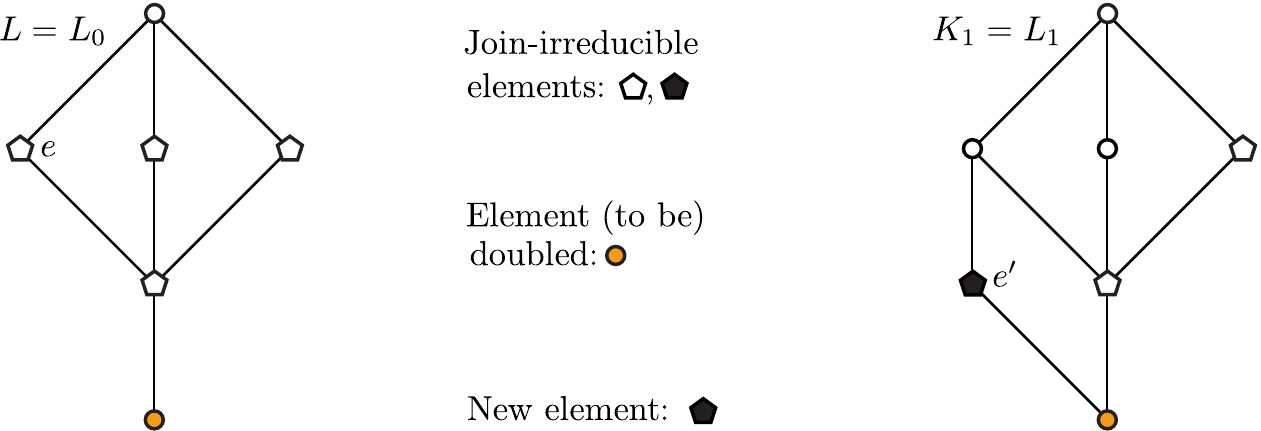}}      %[scale=0.93]{czgamfig1}}
\caption{First step: we lower $e\in\Jir L$ down to a new atom, $e'$. (Some ingredients of Figures~\ref{figp4}--\ref{figp6} will only be explained later.) }\label{figp4}
\end{figure}

\begin{figure}[h]
\centerline
{\includegraphics[width=\textwidth]{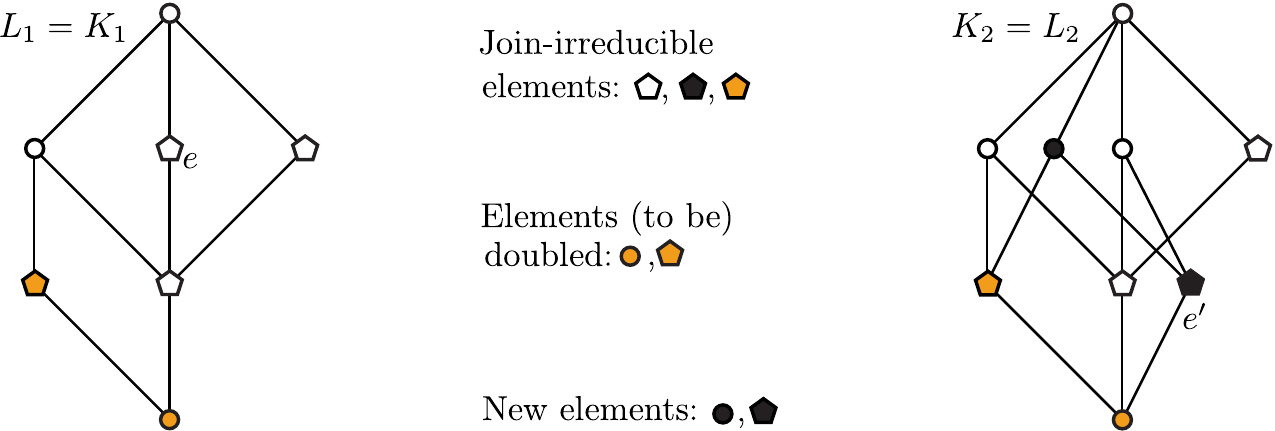}}      %[scale=0.93]{czgamfig1}}
\caption{Second step: another join-irreducible element, also denoted by $e$, is lowered}\label{figp5}
\end{figure}

\begin{figure}[h]
\centerline
{\includegraphics[width=\textwidth]{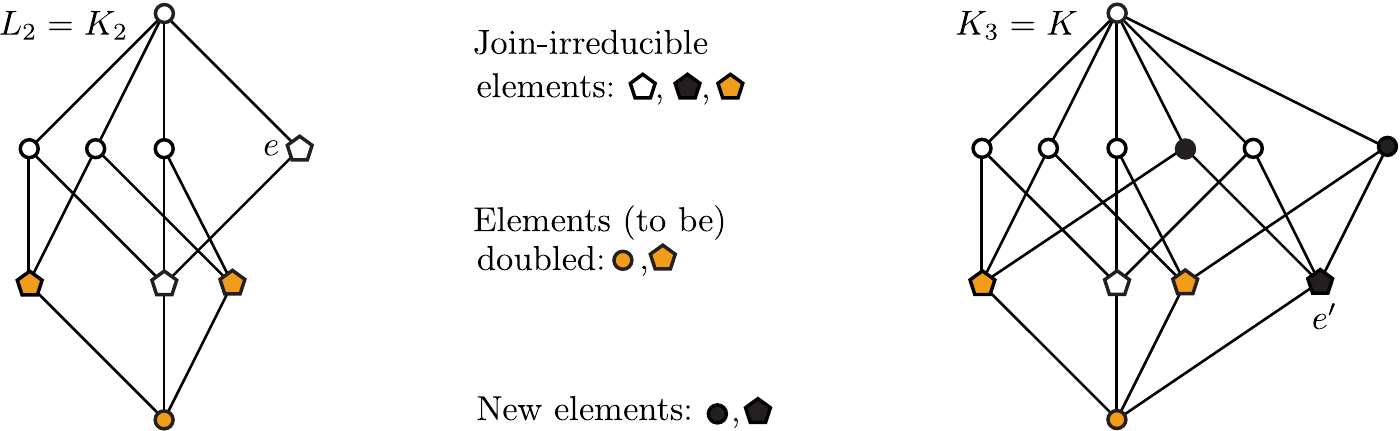}}      %[scale=0.93]{czgamfig1}}
\caption{By lowering the only non-atom join-irreducible element $e$, we obtain the required $K$}\label{figp6}
\end{figure}

\begin{figure}[h]
\centerline
{\includegraphics[scale=1.0]{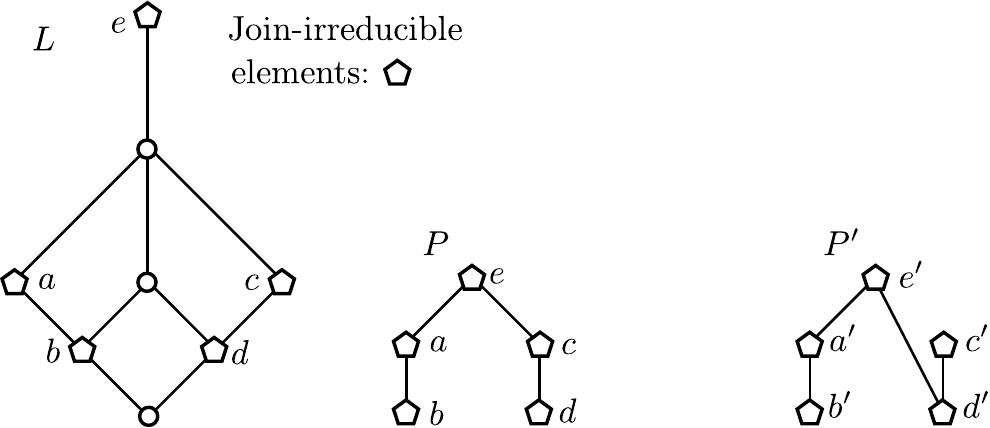}}      %[width=\textwidth]{czglpextfig25}}
\caption{{The ordering of $P=(\Jir L,\wk\rho)$ cannot be reduced to $P'=(\Jir K,\tau)$}}\label{figparx25}
\end{figure}

We admit that after Figures~\ref{figp4}--\ref{figp6}, the concept of lowering is not yet clear. At first sight, the following possibility could offer itself. We could say that $e\in\Jir L$ gets lower if, in an extension $K$ of $L$, there are less join-irreducible elements below $e$ than in the original lattice $L$. For example, apart from renaming $e$ to $e'$,
the only difference between $P$ and $P'$ in Figure~\ref{figparx25} is that
$\set{a,b,c,d}=\set{x\in P: x<_P e}\supset \set{x\in P': x<_{P'}e'}=\set{a,b,d}$. In fact, if $y\neq e$ and $y\in P$, then for any $x\in P$, $x<_P y\iff x<_{P'} y$. However, the following observation, which will be proved at the end of Section~\ref{sect:props}, shows that this meaning of lowering has not much to do with the present paper.

\begin{observation}\label{observe:nlkrttHgb} With $L$ and $P=\Jir L$ given in Figure~\ref{figparx25}, there is no length-preserving semimodular extension $K$ of $L$ such that $\Jir K$ is order-isomorphic to $P'$, given in the same figure.
\end{observation}

\section{Our goals, a little history, some definitions, and an outline}\label{sect:intro}

The reader is only assumed to have some basic familiarity with lattices. For readers knowing that any semimodular lattice of finite length  satisfies the 
\begin{equation}
\parbox{9cm}{\emph{Jordan--H\"older chain condition} (JHCC),  that is, any two maximal chains in an interval have the same length,}
\end{equation}
the paper is probably self-contained. 
We always assume that 
\begin{equation}
\text{every lattice in the paper is a semimodular lattice of finite length}
\label{txt:convSMLFL}
\end{equation}
even when this is not repeated. For brevity, we often call these lattices \emph{\eqref{txt:convSMLFL}-lattices}.

\subsection{First goal: extending the concept of Faigle geometries}
\label{subsect:fgF} For a  \eqref{txt:convSMLFL}-lattice $L$,  $\Jir L$ will always stand for the poset (partially ordered set) of (nonzero) join-irreducible elements of $L$. A 
finite semimodular lattice $L$ is  satisfactorily described by a finite geometry on $\Jir L$, which we call the \emph{Faigle geometry} associated with $L$; see Faigle~\cite{faigle}, and see also Cz\'edli~\cite{czg:faigle} for a recent revisiting. By extending U.\ Faigle's 1980 result, our first goal is to define a larger class of geometries that are in a ``canonical '' bijective correspondence with  \eqref{txt:convSMLFL}-lattices.

\subsection{Second (main) goal: lowering a join-irreducible element}\label{subsect:scglw} 
For a  \eqref{txt:convSMLFL}-lattice $L$,
let $\At L:=\set{a\in L: 0\prec a}$ denote the \emph{set of atoms} of $L$. If, in addition to \eqref{txt:convSMLFL}, $\Jir L=\At L$, then $L$ is a \emph{geometric lattice}. 
By the Dilworth Embedding Theorem, each finite lattice $L$ can be embedded into a finite geometric lattice $K^\ast$. {(Later, Pudl\'ak and T\r uma~\cite{pudlaktuma} proved even more.)}{} Wild~\cite{wild} observed that, \emph{implicitly}, the \emph{proof} of this theorem yields a length-preserving embedding if $L$ is semimodular; see pages 125--131 in P.\ Crawley and R.P.\ Dilworth~\cite{crawleydilworth}. Prior to M.\ Wild's observation,
 Gr\"atzer and Kiss~\cite{ggkiss} proved in 1986 that each finite semimodular lattice $L$ has a \emph{length-preserving embedding} (in other words, a \emph{cover-preserving $\set{0,1}$-embedding}) into a finite geometric lattice $K^\ast$. 
In 1993, Wild~\cite{wild} strengthened this result by constructing a length-preserving extension $K^\ast$ of $L$ with the additional property that $|\At {K^\ast}|=|\Jir L|$. 
In 2010, Cz\'edli and Schmidt~\cite{czgscht2geom} went even further by allowing  \eqref{txt:convSMLFL}-lattices. The proofs of these results are either long and complicated or, as in Wild~\cite{wild}, rely on results from matroid theory. The geometric lattice $K^\ast$ is much larger than $L$ in general, and some of us (including the author) do not have a satisfactory insight into its structure.  This is why our goal is to approach the geometric lattice $K^\ast$ in many steps, by lowering only one non-atom join-irreducible element and lowering it just a little at each step. Theorem~\ref{thmmain}  asserts that such a lowering is always possible. 

\subsection{Third goal: applications and properties of the lowering construction}\label{subsect:thgppl}
The advantage of our lowering construction is 
that it is  simpler than the ``big jump'' from $L$ to the geometric lattice $K^\ast$. 
Due to this fact, we prove Theorem~\ref{thm:infrect}, which is a general result  on length-preserving extensions of semimodular lattices. 
Also, we  present some lemmas and draw some diagrams to give more insight into our construct. The two theorems of the paper imply the mentioned results on extensions into geometric lattices and an unpublished result of  E.\ T.\ Schmidt on higher dimensional rectangular extensions. Our approach to these results is shorter than the original ones except possibly for Wild~\cite{wild}.
Even the classical result that $\length L=|\Jir L|$ for a finite distributive lattice $L$ becomes a trivial consequence of our theorems.

\subsection{Outline} Section~\ref{sect:faigle} generalizes Faigle geometries in \eqref{pbx:RdGvBlJdX} and proves their simplest properties, see Lemma~\ref{lemma:nFjRjf}. Proposition~\ref{prop:sDlspFl} in Section~\ref{sect:ccorsp} establishes a canonical bijective correspondence between our geometries and \eqref{txt:convSMLFL}-lattices.   Section~\ref{sect:lower}  gives the exact meaning of lowering a join-irreducible element, see Definition~\ref{def:lower}, presents a related construction in Definition~\ref{def:Lext}, and proves  Theorem~\ref{thmmain} to ensure that our construction always gives a lowering. Section~\ref{sect:thlms} proves a general result, Theorem~\ref{thm:infrect}, on length-preserving extensions of \eqref{txt:convSMLFL}-lattices. Section~\ref{sect:corol} derives some earlier results from our theorems. Section~\ref{sect:props} proves some properties of the lowering construction and derives two earlier results as corollaries.

\section{A generalization of Faigle geometries}\label{sect:faigle}
There is an important but rarely used bijective correspondence between finite semimodular lattices and Faigle geometries; see Faigle~\cite{faigle}, in which these geometries are called proper geometries, and see Cz\'edli~\cite{czg:faigle}, where these geometries are revisited. In spirit, we go after \cite{czg:faigle}. This section is self-contained, so the reader need not look into \cite{czg:faigle} and  Faigle~\cite{faigle}.
The plan for this section is very simple. Starting from a \eqref{txt:convSMLFL}-lattice $L$ and motivated by Faigle geometries, we define a geometry $\Geom L$ in a natural way. We prove some properties of $\Geom L$ and make these properties axioms. In the next step, we prove some consequences of the axioms. Finally, we prove that $\Geom L$  determines $L$ up to isomorphism.

An element $p$ of a  \eqref{txt:convSMLFL}-lattice $L$ belongs to the poset $\Jir L=(\Jir L, \leq)$ of \emph{join-irreducible elements} if $p$ has exactly one lower cover, which we denote by $\lcov L p$. For a poset $P=(P,\leq_P)$ and $u\in P$, we let
\begin{equation*} 
\parbox{9.5cm}{
$\sideal P u := \set {x\in P:  x\leq u}$,
$\,\,\sodeal P u := \set {x\in P:  x< u}$, and, dually,
$\sfilter P u := \set {x\in P:  x\geq u}$,\quad   
$\sfolter P u := \set {x\in P:  x> u}$.
}
\label{eq:notUDarrows}
\end{equation*}
Note that if, say $u\in P_1$ and $P_1$ is a subposet of $P_2$ then, say, $\sodeal{P_1}u$ is different from $\sodeal {P_2}u$ in general. This is why we rarely omit the subscript from, say, $\sodeal P u$.

\begin{definition}\label{def:L2G} For a semimodular lattice $L$ of finite length, let $P_L$ be the poset $\Jir L$ and let $\alg F_L:=\set{P_L\cap\sideal L x: x\in L}$.
Then $\wha F_L:=(P_L,\alg F_L)$, also denoted by $\Geom L$, is the \emph{geometry of $L$}. The elements of $\alg F$ are called the \emph{flats} of $L$.
\end{definition}

For a set $U$, we denote the \emph{powerset} $\set{X:X\subseteq U}$ of $U$ by $\Pow U$. 
For a chain $C$, $\length C$ is defined by $|C|=1+\length C$. For a poset $P$, $\length P:=\sup\{\length C: C$ is a chain in $P\}$. If $X\subseteq P$ such that for every $u\in X$ we have that $\sideal P u\subseteq X$, then $X$ is called a \emph{down-set} of $P$.  For $x,y\in P$, $x\prec_P y$ denotes that $y$ covers $x$ in $P$. 
Most ingredients of the following lemma and its proof are extracted from Faigle~\cite{faigle}, and almost all ingredients from Cz\'edli~\cite{czg:faigle}.

\begin{lemma}\label{lemma:fhbLjmkLpR}
For a semimodular lattice $L$ of finite length, let $\wha F=(P,\alg F)$ stand for $\Geom L=(P_L,\alg F_L)$; see Definition~\ref{def:L2G}. Then $\wha F$ has the following properties. 
\begin{enumerate}
\item[\textup{(FL)}] $P$ is a poset, $\alg F\subseteq \Pow P$, and the poset $(\alg F,\subseteq)$ is of finite length.
\item[\textup{(F$\cap$)}] $P\in\alg F$ and  $\alg F$ is $\bigcap$-closed, that is, for all $\alg U\subseteq \alg F$, we have that $\bigcap\alg U\in\alg F$; 
\item[\textup{(F$\dwn$)}] Every member of $\alg F$ is a down-set of $P$;
\item[\textup{(Pr)}] $\emptyset\in \alg F$ and for each $u\in P$, both $\sideal P u$ and $\sodeal P u $ belong to $\alg F$; and
\item[\textup{(CP)}] For any $q\in P$ and $X\in \alg F$ such that $q\notin X$ and $\sodeal P q\subseteq X$, there exists a $Y\in \alg F$ such that $X\fprec Y$ and $q\in Y$.
\end{enumerate} 
\end{lemma}

\begin{proof} Since $L$ is of finite length, it is a complete lattice. Hence, we can apply the rule $\bigcap\set{P\cap \sideal L {x_i}: i\in I}=P\cap \sideal  L {\bigwedge\set{{x_i}: i\in I}}$ to obtain that (F$\cap$) holds in  $\wha F$.  Trivially, so does (F$\dwn$).  Since
\begin{equation}
\text{every $z\in L$ is of the form $\sbigvee L (P\cap \sideal L z)$,}
\label{txt:szLvhHmwR}
\end{equation}
it follows that $(\alg F,\subseteq)$ is isomorphic to $(L,\leq)$. Thus, $\wha F$ satisfies (FL). Clearly, $\emptyset=P\cap\sideal L 0\in \alg F$ and, for any $u\in P$, $\sideal P u=P\cap\sideal L u\in\alg F$. By the join-irreducibility of $u$, $\sodeal P u=P\cap \sodeal L u=P\cap\sideal L{\lcov L u} \in\alg F$, whence $\wha F$ satisfies  (Pr). 
To deal with (CP), assume that $X\in \alg F$ is witnessed by $X=P\cap\sideal L x$. Let  $u\in P\setminus X$ such that $\sodeal P u\subseteq X$.
From $\sodeal P u=P\cap\sideal L{\lcov L u}$, we obtain that  $\lcov L u\leq x$ while $u\not\leq_L x$ since $u\notin X$. Hence, 
$x=x\vee \lcov L u \prec_L x\vee u:=y$ by semimodularity. Then $u\in Y:=P\cap\sideal L y$, and $X\fprec Y$ since $(\alg F,\subseteq)\cong(L,\leq)$.
Hence, (CP) holds in $\wha F$, completing the proof.
\end{proof}

Unless otherwise specified explicitly, the meaning of a \emph{geometry} in this paper is always the following.
\begin{equation}
\parbox{6.8cm}{A geometry is a pair $\wha F=(P,\alg F)$ 
satisfying (FL), (F$\cap$), (F$\dwn$), (Pr), and (CP).}
\label{pbx:RdGvBlJdX}
\end{equation}
If confusion threatens or just for emphasis, we can occasionally say that $\wha F$ is a 
\emph{\eqref{pbx:RdGvBlJdX}-geometry}.

From now on, we call (FL), (F$\cap$), (F$\dwn$), (Pr), and (CP) \emph{axioms} rather than properties. 
Note that the acronyms
(Pr) and (CP), taken from Cz\'edli~\cite{czg:faigle}, and (FL) come from``principal'', ``covering property'', and ``finite length,  respectively.
For a geometry $\wha F$, it follows from (Pr) that $\alg F$ is finite if and only if so is $P$; in this case we say that the geometry is finite. 
Based on Cz\'edli~\cite{czg:faigle}, note that finite \eqref{pbx:RdGvBlJdX}-geometries are the same as Faigle geometries, which were called proper geometries in Faigle~\cite{faigle}.

Before establishing a bijective correspondence between 
\eqref{txt:convSMLFL}-lattices and {\eqref{pbx:RdGvBlJdX}-geometries}, it is reasonable to explore some properties of these geometries. Observe that axiom (F$\cap$) means that $\alg F$ is a \emph{closure system}, in other words, a \emph{Moore family} on $P$. As usual in case of closure systems, we can take the \emph{closure map} (in other word, the \emph{closure operator}) $\mclf$ associated with $\alg F$ as follows.
\begin{equation}
\mclf\colon\Pow P\to \Pow P\quad\text{is defined by}\quad
X\mapsto \bigcap\set{Y\in\alg F: X\subseteq X}.
\end{equation}
It is well known (and easy to see) that for any $X\subseteq Y\in\Pow P$, we have that $X\subseteq \clf X=\clf{\clf X}\subseteq \clf Y$; we will frequently use these properties of closure maps implicitly.
For Faigle geometries, the following lemma is known from Cz\'edli~\cite{czg:faigle} and, apart from a slight difference in definitions,  also from Faigle~\cite{faigle}.

\begin{lemma}\label{lemma:nFjRjf}
Let $\wha F=(P,\alg F)$ be a \eqref{pbx:RdGvBlJdX}-geometry,  and assume that $X,Z\in \alg F$. Then
\begin{enumeratei}
\item[\textup{(DC)}] $X\fprec Z\,\,\iff \,\,(\exists u\in P\setminus X)\,\bigl(\sodeal P u\subseteq X\text{ and }Z=\clf{\set u\cup X}\bigr)$. \text{ Also,} 
\item[\textup{(UC)}] in \textup{(CP)}, $Y$ is uniquely determined. Furthermore,
\item[\textup{(CE)}]  $X\fpreceq Z\,\,\iff \,\,(\exists u\in P)\,(\sodeal P u\subseteq X\text{ and }Z=\clf{\set u\cup X})$.
\end{enumeratei}
\end{lemma}
The acronyms above come from ``description of covering'',  ``unique cover'', and ``covers or equals'', respectively. 
Note  that in this paper ``$\subset$''   always means the conjunction of 
``$\subseteq$'' and ``$\neq$''.

\begin{proof}[Proof of Lemma~\ref{lemma:nFjRjf}]
If both $Y_1$ and $Y_2$ satisfied the requirements of (CP) and $Y_1\neq Y_2$, then  $Y_1\cap Y_2\in \alg F$, $X\fprec Y_1$, and $X\fprec Y_2$ 
 would lead to $q\in Y_1\cap Y_2=X$, a contradiction. Thus, we conclude (UC). 

To prove the ``$\Rightarrow$'' part of (DC), assume that $X\fprec Z$. Since $P$ is of finite length, we can take a minimal element  $u$ in $Z\setminus X$.
Then $\sodeal P u\subseteq X$ by the minimality of $u$, and $u\notin X$. Using that $X\fprec Z$ and $X\subset \set u\cup X\subseteq
\clf{\set u\cup X}\subseteq \clf Z=Z$, we conclude that $Z=\clf{\set u\cup X}$, as required. Conversely, to prove 
``$\Leftarrow$'' part, assume the existence of a $u\in P\setminus X$ such that $\sodeal P u\subseteq X$ and $Z=\clf{\set u\cup X}$. By (CP) and (UC), there is a unique $Y\in \alg F$ such that
$u\in Y$ and $X\fprec Y$. Since $X\subset \set u\cup X\subseteq \clf{\set u\cup X}\subseteq \clf Y=Y$, we obtain that $X\fprec Y=\clf{\set u\cup X}=Z$, as required.
This proves (DC). Finally, (CE) is a trivial consequence of (DC) since $u\in X\iff\clf{\set u\cup X}=X$.
\end{proof}

\section{Canonical correspondence}\label{sect:ccorsp} 
This section elaborates the correspondence between \eqref{txt:convSMLFL}-lattices and \eqref{pbx:RdGvBlJdX}-geometries. We need the following definition.
\begin{definition}
For geometries $\wha {F_1}=(P_1,\alg F_1)$ and $\wha {F_2}=(P_2,\alg F_2)$, by an \emph{isomorphism} $\ogmap\colon 
\wha {F_1} \to \wha {F_2}$ we mean a map $\ogmap \colon \alg F_1\to\alg F_2$ defined by $\ogmap(X):=\set{\gmap(y):y\in X}$ where $\gmap\colon P_1\to P_2$ is a poset isomorphism.
\end{definition}
Next, we define to maps; we will show later that their codomains are correctly given.

\begin{definition}
Let $\Gclass$ and $\Lclass$ be the class of 
\eqref{pbx:RdGvBlJdX}-geometries and that  of  \eqref{txt:convSMLFL}-lattices, respectively. (The superscript ``c'' comes from ``class''.)  We consider $\mgeom$, which is given in Definition~\ref{def:L2G},  a map $\mgeom \colon\Lclass\to \Gclass$. Also, we define a map $\mlat\colon \Gclass\to \Lclass$ by the rule that for $\wha F=(P,\alg F)\in\Gclass$, we let $\Lat{\wha F} :=(\alg F,\subseteq)$, it is the \emph{lattice} associated with $\wha F$. 
\end{definition}

\begin{proposition}\label{prop:sDlspFl}
For $L\in\Lclass$ and $\wha F=(P,\alg F)\in\Gclass$,  we have that $\Geom L\in \Gclass$, $\Lat {\wha F}\in \Lclass$, $\Lat{\Geom L}\cong L$, and $\Geom{\Lat{\wha F}}\cong \wha F$.
\end{proposition}

Before the proof, note that  Proposition~\ref{prop:sDlspFl} allows us to say that  \eqref{txt:convSMLFL}-lattices and 
\eqref{pbx:RdGvBlJdX}-geometries are \emph{canonically equivalent}. Also, we can use $\Geom L$ if $L\in\Lclass$ if given, and we can define an $L\in\Lclass$ by defining an 
$\wha F\in \Gclass$ and letting $L:=\Lat{\wha F}$.
It will frequently be implicit that we pass from geometries to lattices and back. 
 The advantage of using geometries is well shown by the proof of Theorem~\ref{thmmain} since a poset defined by a geometry is  automatically a lattice and it is semimodular.

\begin{proof}[Proof of Proposition~\ref{prop:sDlspFl}]
If  $L\in\Lclass$, then $\Geom L\in\Gclass$ by Lemma~\ref{lemma:fhbLjmkLpR}. Next, let $\wha F\in \Gclass$; we need to show that $\Lat {\wha F}=(\alg F,\subseteq)$ belongs to $\Lclass$. Since the members of any closure system are known to form a complete lattice, $\Lat {\wha F}$ is a complete lattice. Note that for $X,Y\in \Lat {\wha F}$, that is, for $X,Y\in\alg F$, we have that $X\wedge Y=X\cap Y$ and $X\vee Y=\clf{X\cup Y}$. By (FL),  $\Lat {\wha F}$ is of finite length. To prove that it is semimodular, assume that $X,Y \in \alg F$ such that $X\wedge Y\fprec X$; we need to show that $Y\fprec X\vee Y$. 
Clearly, $Y\subset X\vee Y$. (DC), see Lemma~\ref{lemma:nFjRjf}, allows us to pick a $u\in P\setminus(X\cap Y)$ such that 
$\sodeal P u\subseteq X\cap Y$ and $X=\clf{\set u\cup(X\cap Y)}$. Observe that $X\vee Y=\clf{X\cup Y}= \clf{\clf{\set u\cup(X\cap Y)}\cup Y}=\clf{ \set u\cup(X\cap Y)\cup Y}=\clf{ \set u\cup Y}$. If $u\in Y$, then $X=\clf{\set u\cup(X\cap Y)}\subseteq \clf Y=Y$ leads to $X=X\wedge Y$, a contradiction. Hence,  $u\notin Y$. Also, $\sodeal P u\subseteq X\cap Y\subseteq Y$, whence (DC) yields that $Y\fprec \clf{ \set u\cup Y}= X\vee Y$, as required. This shows that  $\Lat {\wha F}\in \Lclass$. By \eqref{txt:szLvhHmwR}, it is clear that the map

\begin{equation}
\liso\colon L\mapsto \Lat{\Geom L}\text{ defined by }
x\mapsto \Jir L\cap\sideal L x
\label{eq:lSlkpzs}
\end{equation}
is a lattice isomorphism. Using the notation
$\wha{F'}=(P',\alg F'):= \Geom{\Lat{\wha F}}$, let 
\begin{equation}
\ogiso\colon \wha F\to \wha{F'}
\text{ where }\giso\colon P\to P'\text{ is defined by }u\mapsto \sideal P u.
\label{eq:mgrzhBklpSJ}
\end{equation}
In a straightforward way, by the same argument as in the proof of Cz\'edli~\cite[Theorem 2.5(D)]{czg:faigle}, it follows that $\ogiso$ is an isomorphism. 
\nothing{(Note that the present paper does not need this fact.)}
\end{proof}

\section{Lowering a join-irreducible element}\label{sect:lower}

In addition to \eqref{pbx:RdGvBlJdX}-geometries, the basic concept of the paper is given in the following definition; this concept is visually illustrated by  
Figures~\ref{figp3}--\ref{figp6} and \ref{fig1}--\ref{fig2}; see Figure~\ref{fig1} for a particularly enlightening illustration.
Recall that for $e\in \Jir L$, the unique lower cover of $e$ in $L$ is denoted by $\lcov L e$. For $e'\in \Jir K$, we write
$\lcov K{e'}$ rather than $\lcov K{{e'}}$ or $\lcov K{(e')}$.
When we intend to ``lower'' $e$ to a new join-irreducible element $e'$ of an extended lattice $K$, to be constructed, we would like to decide where $e'$ should be in $K$. A reasonable way to do so is to specify which element of $L$ should be $\lcov K {e'}$.

\begin{figure}[h]
\centerline
{\includegraphics[width=\textwidth]{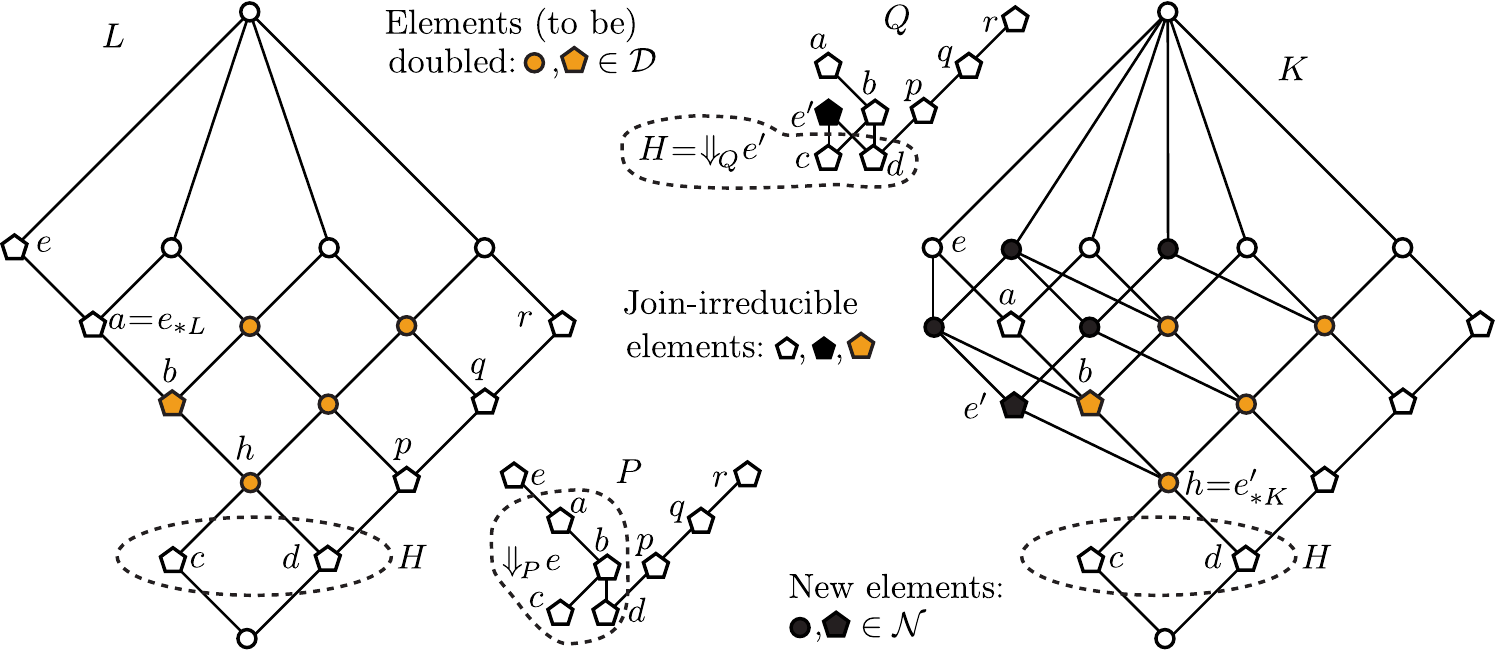}}      %[scale=0.93]{czgamfig1}}
\caption{Illustrating the lowering construction $\Lext$; see Definition~\ref{def:Lext}}\label{fig1}
\end{figure}

\begin{definition}\label{def:lower} 
For a \eqref{txt:convSMLFL}-lattice $L$, $e\in \Jir L$, and a lattice $K$, we say that  \emph{$K$ is obtained from $L$ by lowering $e$} if the following hold.
\begin{enumeratei}
\item\label{def:lowera} $K$ is also a \eqref{txt:convSMLFL}-lattice, $L$ is a sublattice of $K$, and $\length K=\length L$.
\item\label{def:lowerb} With the notation $P:=\Jir L$ and $Q:= \Jir K$, we have that $P\setminus Q=\set e$ and  $Q\setminus P=\set{e'}$ for a unique $e'\in Q$.
\item\label{def:lowerc} With $h:=\lcov K{e'}$, we have that $h\in L$ and  $h<\lcov L e$.
\end{enumeratei}
In this case, if we want to say more about the relation between $K$ and $L$, we use the terminology that $K$ is \emph{an extension of $L$ that lowers $e$ to a cover of $h$}.
\end{definition}

For example, $L$ and $K$ in  Figure~\ref{fig1} show how we can lower $e\in \Jir L$ to a cover of $h\in L$; the dotted ovals, $P$, and $Q$ are not relevant at present. 
We are going to point out soon, right before the proof of Theorem~\ref{thmmain},  why $L$ and $K$ in the figure are semimodular lattices. If $e$ is an atom (that is, $0\prec_L e$), then $h<\lcov L e=0$ is impossible and $e$ cannot be lowered. We are going to prove that if $e\in\Jir L$ is \emph{not} an atom, then there exists an  extension  $K$ of $L$ that lowers $e$ to a cover of $h$. First, in harmony with Figure~\ref{fig1}, we describe a construction. (The rest of figures are also relevant but Figure~\ref{fig1} indicates generality better.)
The set of atoms of $L$ will be denoted by $\At L$; note that $\At L\subseteq \Jir L$.

\begin{definition}\label{def:Lext} 
Let $L$ be a semimodular lattice of finite length, $e\in\Jir L\setminus\At L$,  $h\in L$, and $h<\lcov L e$. We define a lattice $\Lext$ as follows. With
\begin{align}
&D:=\set{x\in \sfilter L h: (\forall y\in L)\,(x\preceq_L y\then e\not\leq y)}\text{. Naturally, we assume}
\label{eq:wmVkrjfNbBj}\\
&\text{that }
N:=\set{\lift x: x\in D}\text{ is disjoint from $L$, and we let }K:=L\cup N.
\label{eq:hmRsnCscPrmTs}
\end{align}
For $y=\lift x\in N$, we often denote $x$ by $\alit y$.
We define a relation $\leq_K$ on $K$ by 
\begin{equation}x\leq_K y \defiff 
\begin{cases}
x,y\in L\text{ and }x\leq_L y&\text{\quad or} \cr
x,y\in N\text{ and }\alit x\leq_L \alit y& \text{\quad or}\cr 
x \in L\text{, } y \in N\text{, and }x\leq_L \alit y& \text{\quad or}\cr
x \in N\text{, }y\in L\text{, and }\alit x\svee L e\leq_L y.
\end{cases}
\label{eq:zSznJznKgTsh}
\end{equation}
Finally, we define $\Lext$ as $(K,\leq_K)$. (It needs a proof that it is a lattice.)
\end{definition}

For $x\in D$, $\lift x$ is a new element, the ``lifted variant'' of $x$. The set $D$ is the set of elements that will be ``\emph doubled'' while $N$ is the set of the ``\emph new'' elements; this explains their notations. 
Since this paper is about easy elementary steps, note the following
\begin{remark}In this paper, for any application of the concept of lowering a join-irreducible element, it suffices to consider the particular case where $h\prec_L \lcov L e$; see Definition~\ref{def:lower}\eqref{def:lowerc}. In this particular case, the construction given in Definition~\ref{def:Lext} is a bit simpler.
\end{remark}

Based on Definitions~\ref{def:lower} and \ref{def:Lext}, we are in the position to formulate the first of the two theorems of the paper.
\begin{theorem}\label{thmmain} Assume that $L$ is  a semimodular lattice of finite length, $e$ is a join-irreducible element of $L$ that is not an atom of $L$, and $h\in L$ is less than the unique lower cover $\lcov L e$ of $e$ in $L$.
Then $K:=\Lext$ from  Definition~\ref{def:Lext} is 
an extension of $L$ that lowers  $e$ to a cover of $h$.
\end{theorem}

The assumption that $e\notin\At L$ is only for emphasis; indeed, this assumption follows from another one, $h<\lcov L e$. Although most applications of Theorem~\ref{thmmain} do not need the following lemma, we formulate this lemma here since we can  economically  prove it jointly with the theorem.

\begin{lemma}\label{lemma:meet} 
With the notation and assumptions of Theorem~\ref{thmmain}, Definitions~\ref{def:lower}, and Definition  \ref{def:Lext}, we have that  $e':=\lift h$ and  
\begin{equation}x\wedge_K y = 
\begin{cases}
x\wedge_L y,&\text{if }x,y\in L,\cr
\lift{(\alit x\wedge_L \alit y)},&\text{if }x, y\in N,\cr
x=\lift{(\alit x\wedge_L y)},&\text{if }x\in N\text{ and  }e\leq y\in L,\cr
\alit x\wedge_L y,&\text{if }x\in N\text{ and  }e\not\leq y\in L,\cr
y=\lift{( x\wedge_L \alit y)},&\text{if }e\leq x\in L  \text{ and  } y\in N,\cr
x\wedge_L \alit y,&\text{if }e\not\leq x\in L  \text{ and  } y\in N
\end{cases}
\label{eq:hbrDtmNklmDnSz}
\end{equation} 
for any $x,y\in K=\Lext$.
\end{lemma}

Before proving the theorem and the lemma above, note that the theory of planar semimodular 
lattices\footnote{{see \texttt{http://www.math.u-szeged.hu/\textasciitilde{}czedli/m/listak/publ-psml.pdf} or see 
the appendix of  \texttt{http://arxiv.org/abs/2107.10202 }}} was essential to find an appropriate construction. 
Also,  it follows immediately from the structural description given  by Cz\'edli~\cite{czg111patchext} 
or   Cz\'edli and Schmidt~\cite{czgschtvisual}, or from the diagrammatic approach developed in Cz\'edli~\cite{czg132diagrrect} (the use of which is exemplified, say, in Cz\'edli~\cite{czglamps}) 
that $L$ in each of our figures is a semimodular lattice. (The semimodularity of $K$ in the same figure follows from Lemma~\ref{lemma:bdscgrpcv}.)

\begin{proof}[Proof of Theorem~\ref{thmmain} and Lemma~\ref{lemma:meet}]
Let $\wha F=(P,\alg F)=\Geom L=(P_L,\alg F_L)$ be the geometry of $L$, see Definition~\ref{def:L2G}. In addition to the usual lattice theoretical tools, the canonical correspondence provided by Proposition~\ref{prop:sDlspFl} allows us to translate some lattice theoretical objects like \eqref{eq:wmVkrjfNbBj}--\eqref{eq:hmRsnCscPrmTs} to the geometrical language, and we can benefit from the axioms (FL)--(CP) as well as from (DC)--(CE) from Lemma~\ref{lemma:nFjRjf}. Let $H:=P\cap\sideal L h\in \alg F$; note that $h=\sbigvee L H$. 
The translation of 
 \eqref{eq:wmVkrjfNbBj} is trivial but we have to use \eqref{eq:zSznJznKgTsh} to rewrite  \eqref{eq:hmRsnCscPrmTs} into \eqref{eq:nszfrkdnzbgrrbmk} below.
(The role of \eqref{eq:zSznJznKgTsh} is to  guarantee that the set theoretical inclusion ``$\subseteq$'' corresponds to ``$\leq_K$''.)
\begin{align}
&\alg D:=\set{X\in \alg F: H\subseteq X\text{ and }
(\forall Y\in\alg F)\,(X\fpreceq Y\then e\notin Y) },
\label{eq:szKvrHnmtpth}
\\
&\alg N:=\set{\set {e}\cup X: X\in \alg D}\quad\text{ and }\quad\alg G:=\alg F\cup \alg N.
\label{eq:nszfrkdnzbgrrbmk}
\end{align}
Observe that (CE) from Lemma~\ref{lemma:nFjRjf} allows us to rewrite the definition of $\alg D$ to 
\begin{equation}
\alg D:=\set{X\in \alg F: H\subseteq X\text{ and }
(\forall u\in P)\,(\sodeal P u\subseteq X\then  e\notin \clf{\set u\cup X } )}.
\label{eq:DfjbtchgrNdwhDsRh}
\end{equation}
For convenience, let us agree that for $X\in\Pow P$, we let 
\begin{equation}
\plue X:=\set e\cup X\quad \text{ and }\quad \mine X:=X\cup \set e.
\label{eq:mplnvNtk}
\end{equation}
Note that for $X\in\alg N$  and $Y\in\alg D$, we have that $\mine X \in \alg D$ and $\plue Y \in\alg N$.
Note also that \eqref{eq:mplnvNtk} harmonizes with the notation $\alit x$ and $\lift y$ used in Definition~\ref{def:Lext}.
The canonical correspondence between $L$ and $\wha F$ will often be used implicitly. 
As opposed to Definition~\ref{def:lower}\eqref{def:lowerb}, we let $R:=P=\Jir L$; at present, it is only an underlying set. We define  a relation  $\leq_R$ on $R$ by
\begin{equation}y<_R x \defiff \left\{
\parbox{3.3cm}{
$y<_P x\neq e$, \quad or \\
$x \in H$ \text{ and } $y=e$.}\right.
\label{pbxDfRrrh}
\end{equation}
If $x,y,z\in R$ such that $x<_R y<_R z$, then we easily obtain that $x<_R z$; either since $e\notin\set{x,y,z}$, or since $H\in \alg F$ is a down-set in $P$ and $H\subseteq \sodeal P e$. Hence, $R=(R,\leq_R)$ is a poset. Observe that 
\begin{equation}
\text{if $X\in\alg F$ and $\sodeal P e\subseteq X$, then $X \notin \alg D$.}
\label{eq:mrnVvlSgrnFrkBrdN}
\end{equation}
Indeed, $H\subseteq \sodeal P e$, $\,|\sideal P e\setminus\sodeal P e|=|\set e|=1$ implies that $\sodeal P e\fprec \sideal P e$, and  $\,\sodeal P e\fprec \sideal P e\ni e$ gives that $\sodeal P e\notin\alg D$. Hence, \eqref{eq:mrnVvlSgrnFrkBrdN} follows from  \eqref{eq:szKvrHnmtpth}.

We claim that $\alg F\cap\alg N=\emptyset$. For the sake of contradiction, suppose that there is a $Y$ belonging to $\alg F\cap\alg N$. Then $Y\in \alg N$ gives that $e\in Y$ and $\mine Y\in\alg D$. However, then $Y\in\alg F$ 
and (F$\dwn$) lead to $\sideal P e\subseteq Y$, whence 
$\sodeal P e\subseteq \mine Y\in\alg D$ contradicts \eqref{eq:mrnVvlSgrnFrkBrdN}. This shows that $\alg F\cap\alg N=\emptyset$. 
A subset $\alg W_1$ of a set $\alg W_2$ of sets is \emph{convex} if for all $X,Z\in \alg W_1$ and $Y\in \alg W_2$, $X\subseteq Y\subseteq Z$ implies that $Y\in\alg W_1$.  We claim that 
\begin{equation}
\parbox{7cm}{$H\in\alg D$, $\alg D$ is a convex subset of $\alg F$, and each of $\alg N$ and $\alg D\cup\alg N$ is a convex subset of $\alg G$.}
\label{eq:hmgRsknlnKsLt}
\end{equation}
To show this, observe that 
 $h<\lcov L e\prec_L e$ gives that $e\notin H$ and (by the canonical correspondence)  the length of the interval $[H,\sideal P e]_{\alg F}$ is at least 2. 
If we had that $e\in  Y$ and $H\fprec Y$ for some $Y\in \alg F$, then (F$\dwn$) would give that $[H,\sideal L e]_{\alg F}$ is a subinterval of the prime interval $[H,Y]_{\alg F}$, which would be a contradiction. Hence, $H\in\alg D$. Clearly, $H$ is the least member of $\alg D\cup \alg N$. 
Thus, to show the convexity of $\alg D$, assume that  $H\subseteq X\subseteq Y$, $X\in\alg F$, and $Y\in \alg D$. If we had a   $u\in P$ with  $\sodeal P u\subseteq X$ and $e\in\clf{\set u\cup X}$, then $Y\in\alg D$ together with  $\sodeal P u\subseteq Y$ and $e\in\clf{\set u\cup Y}$ would contradict \eqref{eq:DfjbtchgrNdwhDsRh}. Hence, again by  \eqref{eq:DfjbtchgrNdwhDsRh}, $X\in\alg D$, showing that  $\alg D
$ is a convex subset of $\alg F$. 
Now assume that $X\subseteq Y\subseteq Z$, $X,Z\in\alg N$, and $Y\in \alg G$. Then $\mine X,\mine Z\in\alg D$.  By $X\subseteq Y$, we have that $e\in Y$, whence $Y=\plue{(\mine Y)}$. Also, $\mine X\subseteq\mine Y\subseteq \mine Z$. The convexity of $\alg D$ gives that $\mine Y\in\alg D$, whereby $Y=\plue{(\mine Y)}\in\alg N$. Thus, $\alg N$ is convex. 
Next,  assume that $H\subseteq X\subseteq Y$, $X\in \alg G$, and $Y\in \alg D\cup\alg N$; we need to show that $X\in \alg D\cup\alg N$. 
If $Y\in \alg D$, then $e\notin Y\supseteq X$ yields that $X\in \alg F$, whereby the convexity of $\alg D$ in $\alg F$ implies that $X\in\alg D\subseteq \alg D\cup\alg N$, as required. Hence, we can assume that $Y\in \alg N$. We can assume that  $X\not\subseteq \mine Y$ since otherwise
$X\in\alg D\subseteq \alg D\cup\alg N$ by the convexity of $\alg D$ again.  Then $e\in X$, but $X\in \alg F$ is impossible since it would lead to $\sodeal P e\subseteq \mine X\subseteq \mine Y \in\alg D$, contradicting \eqref{eq:mrnVvlSgrnFrkBrdN}.  Hence, $X\in \alg G\setminus\alg F=\alg N\subseteq \alg D\cup\alg N$, proving  \eqref{eq:hmgRsknlnKsLt}.
Since $h<\lcov L e$,  $\lcov L e=\sbigvee L (\sideal P{\lcov L e)}= \sbigvee L(\sodeal P e)$, and $h=\sbigvee L H$ yield that 
$H\subset \sodeal P e$, it follows immediately from \eqref{pbxDfRrrh} that, for $u\in R\setminus\set e$ and $v,w\in R$,
\begin{equation}
\parbox{7.0cm}{$\sideal R u=\sideal P u$, \ $\sodeal R u=\sodeal P u$,  \ $\sodeal R e=H\subset\sodeal P e$, \ $\sideal R e=\plue H\subset\sideal P e$, \  and \ $v\leq_R w\then v\leq_P w$.}
\label{eq:smmHrGsGnrsz}
\end{equation}
Hence, $\wha G$ satisfies (F$\dwn$) since so does $\wha F$. We know that $\length{\alg F}=\length L$ is finite. Let $\alg C\subseteq \alg G$ be a chain. Then $|\alg C\cap \alg F|\leq 1+ \length{\alg F}$. Since $\alg N$ and $\alg D$ are order isomorphic (with respect to $\subseteq$),  
$|\alg C\cap \alg N|\leq 1+ \length{\alg D}\leq  1+ \length{\alg F}$. Hence, $|\alg C|\leq 2+2\cdot\length{\alg F}$, showing that (FL) holds for $\wha G$.

Since (F$\cap$) holds for $\wha F$ and $H$ is the least element of $\alg D \cup\alg N$, it follows easily from  
\eqref{eq:hmgRsknlnKsLt}, $H\subseteq \sideal P e$, and 
the validity of (F$\dwn$) for $\wha F$ that, for $X,Y\in\alg G$,   
\begin{equation}X\cap Y = 
\begin{cases}
X\cap Y\in \alg F,&\text{if }X,Y\in \alg F,\cr
\plue{(\mine X\cap \mine Y)}\in \alg N,&\text{if }X,Y\in \alg N,\cr
X= \plue{(\mine X\cap Y)}\in \alg N,&\text{if }X\in \alg N\text{ and  }e\in Y\in \alg F,\cr
\mine X\cap Y\in\alg F,&\text{if }X\in \alg N\text{ and  }e\notin Y\in \alg F,\cr
Y= \plue{(X\cap \mine Y)}\in\alg N,&\text{if }e\in X\in \alg F\text{ and  }Y\in \alg N,\cr 
X\cap \mine Y\in\alg F,&\text{if }e\notin X\in \alg F\text{ and  } Y\in \alg N.
\end{cases}
\label{eq:lspszKlBgLwnqPHrtTs}
\end{equation}
In particular, \eqref{eq:lspszKlBgLwnqPHrtTs} implies that $\wha G$ satisfies (F$\cap$).
Since  $H\in \alg D$ by \eqref{eq:hmgRsknlnKsLt}, we obtain from  \eqref{eq:smmHrGsGnrsz} that $\sodeal R e\in \alg G$ and $\sideal R e=\plue H\in \alg N\subseteq \alg G$. These facts and \eqref{eq:smmHrGsGnrsz} yield that $\wha G$ satisfies (Pr). 

Next, to prove that $\wha G$ satisfies (CP), assume that
\begin{equation} 
\text{$X\in \alg G$, $q\in R\setminus X$, and $\sodeal R q\subseteq X$.}
\label{eq:sZhLrDcSlZkR}
\end{equation} 
We need to find a $Y\in\alg G$ such that $X\gprec Y\ni q$. There are three cases to consider. 

\begin{case}\label{case1} 
In addition to \eqref{eq:sZhLrDcSlZkR}, we assume that $q=e$. Since $e=q\notin X$, we have that $X\in \alg F$. If $X\in\alg D$, then $X\fprec \plue X=:Y\ni e=q$ since $|Y\setminus X|=1$. So assume that  $X\in\alg F\setminus\alg D$. We know that $H=\sodeal R e=\sodeal R q\subseteq X$. This fact, $e\notin X$, and \eqref{eq:szKvrHnmtpth} give a $Y\in \alg F$ with $X\fprec Y\ni e=q$. We only need to show that $X\gprec Y$. For the sake of contradiction, suppose that $X\subset Z\subset Y$ for some $Z\in \alg N$. 
Since $e\notin X\notin \alg D$, $H\subseteq X\subseteq \mine Y\in\alg D$  contradicts \eqref{eq:hmgRsknlnKsLt}. Thus, $X\gprec Y$, as required.
\end{case}

\begin{case}\label{case2} In addition to \eqref{eq:sZhLrDcSlZkR}, we assume that $q\neq e$ and $X\in \alg F$. Then  \eqref{eq:smmHrGsGnrsz} gives that $\sodeal P q=\sodeal R q\subseteq X$. Applying (CP) to $\wha F$, we obtain a $Y\in\alg F$ such that $X\fprec Y\ni q$. 
So it suffices to show that $X\gprec Y$. For the sake of contradiction, suppose that this is not so, and pick a $Z\in\alg N$ such that $X\subset Z\subset Y$. We know from \eqref{eq:nszfrkdnzbgrrbmk} that  $\mine Z\in\alg D$.
If $e\in X$, then $X\in\alg F$ and (F$\dwn$) give that
$\sideal P e\subseteq X$, whence $\sodeal P e\subseteq \mine X\subseteq \mine Z\in\alg D$ contradicts \eqref{eq:mrnVvlSgrnFrkBrdN}. Thus, $e\notin X$, and we have that $X\subseteq \mine Z\subseteq Z\subset Y$. Since  $\mine Z\in\alg D\subseteq\alg F$ and $X\fprec Y$, it follows that 
$X=\mine Z\in\alg D$. But then \eqref{eq:szKvrHnmtpth} and $X\fprec Y$ lead to $e\notin Y$, contradicting $e\in Z\subseteq Y$.  Thus, $X\gprec Y$, as required.
\end{case}

\begin{case}\label{case3} In addition to \eqref{eq:sZhLrDcSlZkR}, we assume that $q\neq e$ and $X\in \alg N$. By \eqref{eq:smmHrGsGnrsz}, $\sodeal P q\subseteq X$. 
Observe that $e\notin\sodeal P q$ since otherwise $\sodeal P e\subseteq \sodeal P q=\sodeal R q\subseteq \mine X\in \alg D$ would contradict \eqref{eq:mrnVvlSgrnFrkBrdN}. Since  $e\notin\sodeal P q$,
 \eqref{eq:sZhLrDcSlZkR} gives that $\sodeal P q=\sodeal R q\subseteq \mine X\in \alg D\subseteq \alg F$ and $q\notin \mine X$. Thus, applying (CP) to $\wha F$, we can pick a $Y\in\alg F$ such that $\mine X\fprec Y\ni q$.  Note that $\alg D\ni \mine X\fprec Y$ and \eqref{eq:szKvrHnmtpth} imply that $e\notin Y$. If $Y\in\alg D$, then \eqref{eq:hmgRsknlnKsLt} easily implies that $X\gprec \plue Y\ni q$, and $\plue Y$ does the job instead of $Y$.
Thus, we can assume that $Y\in \alg F\setminus \alg D$.
Combining this with $H\subseteq \mine X\subset Y$ and  \eqref{eq:szKvrHnmtpth}, we can pick an $E\in \alg F$ such that $Y\fpreceq E\ni e$. In fact, $Y\fprec E$ since $e\notin Y$ but $e\in E$. Observe that  $q\in Y\subseteq E$,   $X=\set e\cup\mine X\subseteq E\cup Y=E$, but $X\neq E$ since $q\in E\setminus X$. That is, $X\subset E\ni q$. Hence, to complete Case~\ref{case3}, we only need to show that $X\gprec E$. For the sake of contradiction, suppose that there is an $S\in\alg G$ such that $X\subset S\subset E$. Note that $e\in S$ since $X\in\alg N$ gives that $e\in X\subset S$. There are two subcases depending on whether $S$ is in $\alg F$ or not. 
First, assume that $S\in\alg F$. Then 
$\mine X\subset X\subset S$ gives that $\mine X\subset S$. By \eqref{eq:szKvrHnmtpth},   $\mine X\in\alg D$ and $e\in S$ exclude that $\mine X\fprec S$. Hence, $\length{[\mine X,S]_{\alg F}}\geq 2$. Combining this with $S\subset E$, we have that 
$\length{[\mine X,E]_{\alg F}}\geq 3$. But this contradicts the semimodularity of $\alg F\cong L$ by (JHCC) since $\mine X\fprec Y\fprec E$.

Second, assume that $S\in\alg N$. Then both $X$ and $S$ contain $e$ since they belong to $\alg N$, whence 
$X\subset S$ turns into $\mine X\subset \mine S$.  
Furthermore, $\mine S\subset S\subset E$, and since  $\mine S\in \alg D$, \eqref{eq:szKvrHnmtpth} shows that $\mine S\not\fprec E$. Combining this with $\mine X\subset \mine S$ similarly to the first subcase, we obtain that $\length{[\mine X,E]_{\alg F}}\geq 3$, and we have already seen how this leads to a contradiction. This completes Case~\ref{case3}.
\end{case}

It follows from Cases~\ref{case1}--\ref{case3} that (CP) holds for $\wha G$ and $\wha G$ is a \eqref{pbx:RdGvBlJdX}-geometry. Let $K:=\Lat{\wha G}$. We are going to show that $K:=\Lext$ and it is an extension of $L$ that lowers $e$ to a cover of $h$. 

By Proposition~\ref{prop:sDlspFl}, $K$  is a semimodular lattice of finite length. To show that $\length K=\length L$, it suffices to find a maximal chain in $\alg F$ that is also a maximal chain in $\alg G$. First, take a maximal chain in $X_0\fprec X_1\fprec\dots\fprec X_t=\sodeal P e$ in the interval $[\emptyset,\sodeal P e]_{\alg F}$. Since $e\notin X_i$ for $i=0,1,\dots,t$, no member of $\alg N$ can be inserted into this chain, that is, the coverings here are coverings also in $\alg G$. Since $|\sideal P e\setminus\sodeal P e|=|\set e|=1$, we have that $X_t=\sodeal P e\fprec \sideal P e=:X_{t+1}$ and $X_t\gprec X_{t+1}$. 
Next, in the interval $[\sideal P e,P]_{\alg F}$, we take a maximal chain $\sideal P e=X_{t+1}\fprec \dots \fprec X_k=P$. It follows 
easily from \eqref{eq:mrnVvlSgrnFrkBrdN} that no member of $\alg N$ can be inserted, whence $X_{t+1}\gprec \dots \gprec X_k$. Consequently, $\set{X_0,X_1,\dots, X_k}$ is a maximal chain both in $\alg F$ and $\alg G$, whereby $\length K=\length L$. It was proved by
Wild~\cite{wild} that 
\begin{equation}
\parbox{11cm}{whenever  $L'$ and $K'$ are semimodular lattices of the same finite length and $L'$ is a meet-subsemilattice of $K$', then $L'$ is a sublattice of $K'$;}
\label{pbx:wvdcnvbrdfnkkg}
\end{equation}
see also Cz\'edli~\cite{czg:faigle} for an outline of the proof.
(Note that Wild~\cite{wild} proved even more.)  Since the meet is the intersection both in $\alg F\cong L$ and $\alg G\cong K$,
or because of the first line of \eqref{eq:lspszKlBgLwnqPHrtTs}, Wild's above-mentioned result yields that $L$ is a sublattice of $K$. 
The validity of \eqref{eq:hbrDtmNklmDnSz} follows immediately from the fact that it is only a transcript of \eqref{eq:lspszKlBgLwnqPHrtTs}.

It is clear that $\Jir K$ is
\begin{equation}
\Jir {\alg G}=(\set{\sideal R u: u\in R},\subseteq)\text{; in fact, for any \eqref{pbx:RdGvBlJdX}-geometry }(R,\alg G),
\label{eq:sWzhGrjmGdrs}
\end{equation}
not only for our $\wha G=(R,\alg G)$.
Indeed, $\sodeal R u$ is clearly the unique lower cover of $\sideal R u$, whence the set given in \eqref{eq:sWzhGrjmGdrs} consists of join-irreducible members of $\alg G$, and (F$\dwn$) makes it clear that each member of $\alg G$ is the join (in fact, the union) of \emph{these} join-irreducible members of $\alg G$. Let $Q$ denote the poset $\Jir K$, and let $e':=\sideal R e$.  By the canonical correspondence and \eqref{eq:sWzhGrjmGdrs}, $Q$ is order isomorphic to $R$. 
Since $\sodeal R e=H$ is the only lower cover of $e'$ and $H$ corresponds to $h$,  we have that $\lcov K {e'}=h$. By \eqref{eq:smmHrGsGnrsz},  $\Jir{\alg G}\setminus\set{\sideal R e}=
\Jir{\alg G}\setminus \set{\sideal P e}$. Hence, the canonical correspondence gives that condition \eqref{def:lowerb} of Definition~\ref{def:lower} holds. This completes the joint proof of Theorem~\ref{thmmain} and Lemma~\ref{lemma:meet}.
\end{proof}

\section{A more general theorem on  length-preserving extensions}\label{sect:thlms}

Two chains, $C_1$ and $C_2$, in a poset are \emph{parallel} if $x_1\parallel x_2$ for all $x_1\in C_1$ and $x_2$ in $C_2$. (As usual, $x_1\parallel x_2$ is the conjunction of $x_1\not\leq x_2$ and  $x_2\not\leq x_1$.) Parallel chains are  \emph{set-theoretically disjoint}, that is, their intersection  is $\emptyset$. (In Lemma~\ref{lemma:ngyhjnlvvgn} and thereafter, there will occur a lattice-theoretical concept of disjointness.)  Every poset is the union of a set of pairwise parallel chains; for example, we can take all singleton chains but, usually, there are other possibilities, too. With the help of our $\Lext$ construction, we are going to prove the following result.

\begin{theorem}\label{thm:infrect} 
Let $L$ be a semimodular lattice of finite length. Assume that $\Jir L$ is the union of a set $\set{C_i: i\in I}$  of pairwise \emph{set-theoretically disjoint} chains. Then $L$ has a length-preserving semimodular extension $K$ such that for a set $\set{C_i': i\in I}$ of pairwise \emph{parallel} chains of $\Jir K$, we have that $\Jir K=\bigcup\set{C_i': i\in I}$ and, for all $i\in I$, $|C_i'|=|C_i|$.
\end{theorem}

In other words, if $L$ is a semimodular lattice of finite length and  we partition $\Jir L$ into chains, then for a length-preserving semimodular extension $K$ of $L$, we can partition $\Jir K$ into pairwise parallel (new) chains that are of the same sizes as the original chains.

\begin{remark} Later, Lemma~\ref{lemma:ngyhjnlvvgn} will  allow us to say 
``pairwise  \emph{lattice-theoretically disjoint} chains of $\Jir K$'' instead of ``pairwise \emph{parallel} chains of $\Jir K$'' in 
 Theorem~\ref{thm:infrect}.
\end{remark}

For later reference, note the following.

\begin{remark}\label{rem:kthvtkhsnrfdn}
Let $\Gamma$ be a property of \eqref{txt:convSMLFL} lattices preserved by taking  directed unions and  lowering any join-irreducible element. If $\Gamma$ holds in $L$, then 
the proof below automatically shows that $\Gamma$  also holds in $K$. If $L$ is finite, then directed unions do not occur and it suffices that the lowering construction preserves $\Gamma$.
\end{remark}

\begin{proof}[Proof of Theorem~\ref{thm:infrect}]
Let $L$ be a \eqref{txt:convSMLFL}-lattice, and denote the set (not the poset) $\Jir L$ by $J$. We denote by $\wk\rho =\set{(u,v)\in J\times J: u\leq_L v}$ the ordering of $\Jir L$, that is, $\Jir L=(J,\wk\rho)=P$.  For $i\in I$, let $\rho_i$ denote the ordering of $C_i$, and let $\rho_\ast:=\bigcup\set{\rho_i:i\in I}$. Since  $\set{C_i:i\in I}$ is a partition of $P$, the relation $\rho_\ast$ is an ordering of $P$ and $\rho_\ast \subseteq \wk\rho$. Note that if $\rho_\ast=\wk\rho$, then the chains are pairwise parallel and we are ready by taking $K:=L$ and $C_i':=C_i$, for $i\in I$. Let $\wha F$ be the geometry $\Geom L=((J,\wk\rho),\alg F)$; see Proposition~\ref{prop:sDlspFl}. 
At this point, interrupting the proof of Theorem~\ref{thm:infrect}, we formulate the following auxiliary statement.

\begin{claim}\label{claim:hncSd} If $\rho_\ast\subset\wk\rho$, then there is a geometry $\wha G=((J,\tau),\alg G)$ such that $\rho_\ast\subseteq\tau \subset \wk\rho$ and $\alg G\supset\alg F$. 
\end{claim}

\begin{proof}[Proof of Claim~\ref{claim:hncSd}]
Since $\rho_\ast\subset\wk\rho$ and $\Jir L =(J,\wk\rho)$ is of finite length, we can pick $i,j\in I$,
$e\in C_i$, and $b\in C_j$ such that $i\neq j$, $b<e$,
and $e$ is the least element of $C_i\cap\sfilter L b$. 
Denote $C_i\cup\set{0_L}$ by $C_i^{+0}$. Let $h$ be the largest element of $\sodeal{C_i^{+0}}e=\set{x\in C_i^{+0}: x<e}$; in notation, $h\prec_{C_i^{+0}}e$. We claim that $h<\lcov L e$. Since $e$ is join-irreducible, $b<e$ gives that $b\leq \lcov L e$. So if $h=0$, then the required $h<\lcov L e$ follows from $0<b\leq \lcov L e$. Assume that $h\neq 0$. From $h<e$ and $b<e$, we obtain that $b\leq h\svee L b\leq \lcov L e$. Since $e$ is the least element of  $C_i\cap\sfilter L b$, it follows that $\lcov L e\notin C_i$. But $h\in C_i$ and $h\leq \lcov L e$, whence $h<\lcov L e$. That is, in both possible cases, $h<\lcov L e$. 
Therefore, Proposition~\ref{prop:sDlspFl} and Theorem~\ref{thmmain} allow us to take the geometry $\wha G=((J,\tau),\alg G)$ that corresponds to  $K:=\Lext$.  (If we look into the proof of 
Theorem~\ref{thmmain}, then Proposition~\ref{prop:sDlspFl} is not needed since $K$ is constructed from $\wha G$.) By construction, $\alg G\supset\alg F$. It follows from \eqref{eq:smmHrGsGnrsz} that $\rho_\ast\subseteq \tau\subset\wk\rho$, completing the proof of Claim~\ref{claim:hncSd}.
\end{proof}

Next, resuming the proof of Theorem~\ref{thm:infrect},
define $T$ as the set of all geometries $\wha G=((J,\tau),\alg G)$ such that $\alg F$ is a length-preserving subset of $\alg G$  and $\rho_\ast\subseteq\tau\subseteq \wk\rho$. For members $\wha G=((J,\tau),\alg G)$  and $\wha{G'}=((J,\tau'),\alg G')$ of $T$, we define $\wha G\leq \wha{G'}$ as the conjunction of 
$\alg G\subseteq \alg G'$ and $\tau\supseteq \tau'$. This definition turns $T$ into a poset $(T,\leq)$. Interrupting the proof again, we formulate the following statement.

\begin{claim}\label{claim:zorn} Each chain in $(T,\leq)$ has an upper bound in $(T,\leq)$. 
\end{claim}
 
\begin{proof}[Proof of Claim~\ref{claim:zorn}]
Some trivial details like ``$P\in$'' or ''$\emptyset\in$'' will be omitted. Let $\set{\wha G_i:i\in I}$ be a chain in $T$ where $\wha G_i=((J,\tau_i),\alg G_i)$. We define $\alg G:=\bigcup\set{\alg G_i:i\in I}$ and
$\tau:=\bigcap\set{\tau_i:i\in I}$. We only need to show that 
$\wha G=((J,\tau),\alg G)$ belongs to $T$;  then $\wha G$ is clearly an upper bound of the chain $\set{\wha G_i:i\in I}$.
Clearly,  $\rho_\ast\subseteq\tau\subseteq \wk\rho$, $(J,\tau)$ is a poset, and this poset is of finite length since so is $(J,\wk\rho)$. A \emph{finite} chain in $\alg G$ with more than $1+\length{\alg F}$ elements (which are subsets of $J$) would be a chain in $\alg G_i$ for some $i\in I$ since the union defining $\alg G$ is directed, and this would contradict  $\length{\alg G_i}= \length{\alg F}$. Thus, (FL) holds for $\wha G$. This implies that 
\begin{equation}
\parbox{8cm}{any intersection of members of  $\alg G$ is actually an intersection of finitely many of the given members.}
\label{pbx:ktrnspchbwlpWk}
\end{equation}
Since $\alg G$, as a directed union, is closed with respect to
the binary intersection, \eqref{pbx:ktrnspchbwlpWk} implies that $\wha G$ satisfies (F$\cap$). 
Using  the first of the obvious equalities
%Combining \eqref{pbx:ktrnspchbwlpWk} with the first of the obvious equalities
\begin{equation}
\sideal{(J,\tau)}u=\bigcap\set{\sideal{(J,\tau_i)}u: i\in I}\,\,
\text{ and }\,\, \sodeal{(J,\tau)}u=\bigcap\set{\sodeal{(J,\tau_i)}u: i\in I}
\label{eq:vdlpzsmdksdmgglx}
\end{equation}
together with
$\sideal{(J,\tau_i)}u\in \alg G_i\subseteq \alg G$
(coming from the validity of (Pr) for $\wha G_i$)
 and (F$\cap$) for $\wha G$,
 we conclude that (F$\dwn$) holds for $\wha G$. 
Similarly, so does (Pr) by \eqref{eq:vdlpzsmdksdmgglx}, $\set{\sideal{(J,\tau_i)}u,\,\,\sodeal{(J,\tau_i)}u}\subseteq \alg G_i\subseteq \alg G$ and (F$\cap$).

Next, to deal with (CP), assume that $q\in J$, $X\in \alg G$, $q\notin X$, and $\sodeal{(J,\tau)}q\subseteq X$. 
Pick an $i_0\in I$ such that $X\in \alg G_{i_0}$. By 
 \eqref{pbx:ktrnspchbwlpWk} \eqref{eq:vdlpzsmdksdmgglx}, and $\sodeal{(J,\tau_i)}u\in \alg G_i\subseteq \alg G$,
$\set{i_0}$ can be extended to a finite subset $I'$ of $I$ such that  $\sodeal{(J,\tau)}u=\bigcap\set{\sodeal{(J,\tau_i)}u: i\in I'}$. Pick an $i'\in I$ such that $\wha G_{i'}\geq \wha G_i$ for all $i\in I'$. Then $\tau_{i'}\subseteq \tau_i$ and $\sodeal{(J,\tau_i)}u \supseteq  \sodeal{(J,\tau_{i'})}u$
 for all $i\in I'$. Combining this with the previous equality, $\sodeal{(J,\tau)}u\supseteq \sodeal{(J,\tau_{i'})}u$. Since \eqref{eq:vdlpzsmdksdmgglx} gives the converse inclusion, we have that $\sodeal{(J,\tau)}u = \sodeal{(J,\tau_{i'})}u\in \alg G_{i'}$. Thus, since $X\in\alg G_{i_0}\subseteq \alg G_{i'}$ and $u\notin X$, (CP) applied to $\alg G_{i'}$ yields a $Y\in\alg G_{i'}$ such that $q\in Y$ and $Y$ covers $X$ in $\alg G_{i'}$. 
If we had a $Z\in\alg G$ with $X\subset Z\subset Y$, then we could pick an $i''\in I$ with $\set{Z}\cup\alg G_i\subseteq \alg{G_{i''}}$ and then $\length{\alg G_{i'}} < \length{\alg G_{i''}}$ would contradict the fact that  $\alg G_{i'}$ is a sublattice of  $\alg G_{i''}$ by \eqref{pbx:wvdcnvbrdfnkkg} and both $\alg G_{i'}$ and $\alg G_{i''}$ are semimodular lattices of the same length.
Therefore, $\wha G$ satisfies (CP), completing the proof of 
Claim~\ref{claim:zorn}.
\end{proof}

Resuming the proof of Theorem~\ref{thm:infrect}, Zorn's lemma and Claim~\ref{claim:zorn} allow us to take a maximal element  $\wha G=((J,\tau),\alg G)$ of the poset $T$.  It follows from Claim~\ref{claim:hncSd} that $\tau=\rho_\ast$. Therefore, 
$K:=\Lat{\wha G}$ satisfies the requirements of Theorem~\ref{thm:infrect}, completing its proof.
\end{proof}

\section{Some corollaries}\label{sect:corol}
Before deriving any corollaries, note the following, which should be taken into account when we mention short  proofs.
\begin{remark} For the particular case when $L$ is finite, the proof of  Theorem~\ref{thm:infrect} could be much shorter since Claim~\ref{claim:zorn} and Zorn's lemma would not be necessary. In fact, for a finite $L$, the proof of Theorem~\ref{thm:infrect} reduces to that of Claim~\ref{claim:hncSd}, which we apply in a finite number of steps.
\end{remark}

Theorems~\ref{thmmain} and \ref{thm:infrect}{} imply some earlier results and offer, with at most one exception, easier (usually much easier) proofs of these results than the original ones.

\begin{corollary}[{Gr\"atzer and Kiss~\cite[Lemma 17]{ggkiss}}]\label{cor:ggkiss}
Each finite semimodular lattice has a length-preserving extension into a geometric lattice.
\end{corollary}

The proof of this corollary (together with that of Theorem~\ref{thmmain})  is much shorter than the proof in Gr\"atzer and Kiss~\cite{ggkiss}. The proof of the following corollary is competitive with that given in Wild~\cite{wild}.

\begin{corollary}[{Wild~\cite[Theorem 4]{wild}}]\label{cor:wild}
Each finite semimodular lattice $L$ has a length-preserving extension into a geometric lattice $K$ such that $|\At K|=|\Jir L|$.
\end{corollary}

\begin{proof}[Proof of Corollaries \ref{cor:ggkiss} and \ref{cor:wild}] Apply Theorem~\ref{thmmain}  repeatedly. 
Alternatively, apply Theorem~\ref{thm:infrect} for the 
 partition $\set{\set u:u\in\Jir L}$.
\end{proof}

The proof of the following corollary is also much simpler than the original one.

\begin{corollary}[{Cz\'edli and Schmidt~\cite[Theorem 1]{czgscht2geom}}]\label{cor:czsch}
Each  semimodular lattice $L$ of finite length has a length-preserving extension into a geometric lattice $K$ such that $|\At K|=|\Jir L|$.
\end{corollary}

Note that Skublics~\cite[Corollary 2]{skublics} extended this result  to some semimodular lattices that are not of finite heights, but his result is beyond the scope of the present paper.

\begin{proof}[Proof of Corollary~\ref{cor:czsch}] 
 Apply Theorem~\ref{thm:infrect} for the 
 partition $\set{\set u:u\in\Jir L}$.
\end{proof}

A poset is of \emph{width $k$} if it is the union of $k$ chains but not of fewer chains. For a \eqref{txt:convSMLFL}-lattice $L$ and chains $C_1,C_2\subseteq \Jir L$, we say that $C_1$ and $C_2$ are \emph{lattice-theoretically disjoint} if $x_1\swedge L x_2=0_L$ for all $x_1\in C_1$ and $x_2\in C_2$. Since $0_L\notin \Jir L$, $C_1$ and $C_2$ are parallel in this case. 
Since two parallel chains in $\Jir L$ need not be lattice-theoretically disjoint in $L$, it is worth formulating the following lemma.

\begin{lemma}\label{lemma:ngyhjnlvvgn} Let $L$ be a \eqref{txt:convSMLFL} lattice, and let $\set{C_i: i\in I}$ be a partition of $\Jir L$ into chains.  Then $\set{C_i: i\in I}$ is a set of pairwise parallel chains if and only if it is a set of pairwise lattice-theoretically disjoint chains.
\end{lemma}

\begin{proof} Since two lattice-theoretically disjoint chains in $L\setminus\set{0_L}$ are trivially parallel, it suffices to prove the ``only if'' part. For the sake of contradiction, suppose that $\set{C_i: i\in I}$  consists of pairwise parallel chains but, for some $i',i''\in I$, $i'\neq i''$ and $C_{i'}$ is not disjoint lattice-theoretically from $C_{i''}$. 
Take an  $x'\in C_{i'}$ and an $x''\in C_{i''}$ such that $x'\swedge L x''\not= 0_L$. Then we can pick a $y\in\Jir L$ such that $y\leq_L x'$ and $y\leq_L x''$. There is a unique $j\in I$ such that $y\in C_j$. Since $C_{i'}$ and $C_j$ are parallel chains, $j=i'$. But  $C_{i''}$ and $C_j$ are parallel chains, too, whereby $j=i''$. Hence, $i'=i''$, which is a contradiction.
\end{proof}

\begin{definition}[E.\ T.\ Schmidt, unpublished]\label{def:kdimrect} 
A \emph{$k$-dimensional rectangular lattice} is a finite semimodular lattice $L$ such that $k=\width{\Jir L}$ and $\Jir L$ is the union of $k$ pairwise lattice-theoretically disjoint chains. 
\end{definition}

\begin{corollary}[E.\ T. Schmidt, unpublished]\label{corol:sch}
Let $L$ be a finite semimodular lattice, and let $k:=\width{\Jir L}$. Then $L$ has a  length-preserving extension into a $k$-dimensional rectangular lattice $K$ such that $|\Jir K|=|\Jir L|$.  Furthermore, if  
$\set{C_1,\dots,C_n}$ is a partition of $\Jir L$ into chains, then we can choose $K$ so that $\Jir K=C_1'\cup\dots\cup C_n'$ with pairwise lattice-theoretically disjoint chains $C_1',\dots,C_n'$ satisfying $|C_1'|=|C_1|$, \dots, $|C_n'|=|C_n|$.
\end{corollary}

Note that the particular case $\width{\Jir L}=2$ of this result is due to Gr\"atzer and Knapp~\cite{gratzerknapp3}.

\begin{proof}[Proof of Corollary~\ref{corol:sch}] Note in advance that Remark~\ref{rem:kthvtkhsnrfdn} also applies to this proof. By definition, 
$\Jir L$ is the union of $k$ chains,  $C_1,\dots,C_k$.
We can assume that these chains  are pairwise set-theoretically disjoint. Indeed, if not so, then replacing $C_i$ by $C_i\setminus C_1$ for $i=2,\dots, k$, $C_1$ becomes set-theoretically disjoint from $C_2,\dots,C_k$. In the next step, we replace $C_i$ by $C_i\setminus C_2$ for $i=3,\dots, k$, and  $C_2$ becomes  set-theoretically  disjoint from $C_3,\dots,C_k$. And so on. Now that our chains are pairwise  set-theoretically disjoint, 
Theorem~\ref{thm:infrect} combined with Lemma~\ref{lemma:ngyhjnlvvgn} yields a length-preserving extension of $L$ into a \eqref{txt:convSMLFL}-lattice $K$ with the required property.
\end{proof}

\section{Some properties, examples, and further corollaries  of the lowering construction}\label{sect:props}

For future reference and a better insight into the construct $\Lext$, we prove some of its properties in this section as well as some corollaries.

\begin{lemma}\label{lemma:dist} 
With the assumptions and notations of Definition~\ref{def:Lext} and  Theorem~\ref{thmmain}, if $L$ is distributive, then so is $K=\Lext$.
\end{lemma}

\begin{proof} Due to the canonical correspondence established by Proposition~\ref{prop:sDlspFl}, we can work with $\wha F:=\Geom L=(P,\alg F)$ and, instead of $K$, with $\wha G=(R,\alg G)$ defined by \eqref{eq:szKvrHnmtpth}, \eqref{eq:DfjbtchgrNdwhDsRh}, and \eqref{pbxDfRrrh}.
By the well-known structure theorem of finite distributive lattices, see, for example,  Gr\"atzer~\cite[Theorem 107]{ggfoundbook},
$\alg F$ is the collection of all down-sets of $P=\Jir L$. 
By the same structure theorem and (F$\dwn$), it suffices to show that every down-set of $R$ belongs to $\alg G$. Let $S\subseteq R$ be a down-set of $R$. This means that for every $u\in S$, $\sideal R u\subseteq S$. If $e\notin S$, then $\sideal P u=\sideal R u\subseteq S$ for all $u\in S$ by 
\eqref{eq:smmHrGsGnrsz}, whence $S$ is a down-set of $P$ and we have that  $S\in\alg F\subseteq \alg G$, as required. 

Next, assume that $e\in S$. We can assume that $e$ is a maximal element of $S$ with respect to $\leq_R$ or, equivalently by \eqref{eq:smmHrGsGnrsz}, with respect to $\leq_P$. Indeed, if there is a $u\in S$ with $e<_P u$, then \eqref{eq:smmHrGsGnrsz} gives that  $\sideal P e\subset\sideal P u=\sideal R u\subseteq S$
and, for every $v\in S\setminus \set e$, $\sideal P v=\sideal R v\subseteq S$, whence $S$ is a down-set of $P$ and we obtain that $S\in\alg F\subseteq \alg G$.  So $e$ is a maximal element of $S$. Let $X:=S\setminus\set e=:\mine S$. Since $H=\sodeal R e\subset \sideal R e\subseteq S$, it follows that $H\subseteq X$. Since $X$ is a down-set of $P$ by  \eqref{eq:smmHrGsGnrsz}, $X\in \alg F$. If no cover of $X$ in $\alg F$ contains $e$, then $X\in \alg D$ by \eqref{eq:szKvrHnmtpth} and $S=\plue X\in\alg N\in\alg G$, as required. 
So we can assume that there is a $Y\in \alg F$ such that $X\fprec Y\ni e$.  Since $e\notin X$, we know that $\sideal P e\not\leq X$ in $\alg F\cong L$. 
But $\sideal P e\leq Y$ since $e\in Y$, so $X<X  \svee{\alg F} \sideal P e 
\leq Y$, whence  $X\fprec Y$ yields that $X  \svee{\alg F} \sideal P e = Y$. Transposed intervals are isomorphic in a distributive lattice (and even in a modular lattice), whereby $X \swedge{\alg F} \sideal P e\fprec \sideal P e$. But the only lower cover of $\sideal P e$ in $\alg F$ is $\sodeal P e$ since $\sideal P e$ corresponds to $e\in \Jir L$  by the canonical correspondence. 
Hence, $X\cap \sideal P e=X\swedge{\alg F} \sideal P e=\sodeal P e$. Therefore, $\sodeal P e\subseteq X$ and $\sideal P e\subseteq \plue X = S$. 
The inclusion $\sideal P e\subseteq S$ together with \eqref{eq:smmHrGsGnrsz}, which gives that $\sideal P u=\sideal R u\subseteq S$ for all $u\in S\setminus\set e$, imply that $S$ is a down-set of $P$. Hence, $S\in \alg F\subseteq \alg G$, as required.
\end{proof}

The following statement has recently been proved in Cz\'edli and Molkhasi~\cite[Lemma 4.4]{czgmolkhasi}. The proof here is shorter.

\begin{corollary}[Cz\'edli and Molkhasi~\cite{czgmolkhasi}]\label{corol:czgmolkhasi}
Corollary~\ref{corol:sch} remains true if we replace ``semimodular''' by ``distributive'' at both occurrences.
\end{corollary}

\begin{proof} Combine  Lemma~\ref{lemma:dist} with the first sentence of the proof of Corollary~\ref{corol:sch}.
\end{proof}

Although we will soon explain the concept of join-distributivity occurring in the remark below, the reader can safely skip over these details. 

\begin{example}In general, the  passage from $L$ to $\Lext$ preserves neither modularity, nor join-distributivity. This is witnessed by Figure~\ref{fig2}, where $\set{a,b,c}$ generates a ``pentagon'' $N_5$ in $K:=\Lext$ on the left while it generates an $M_3$ on the right.
\end{example}

A finite semimodular lattice is \emph{join-distributive} if it contains no $M_3$ (five-element
nondistributive modular lattice) as a sublattice. Equivalently, if it contains no $M_3$ as a cover-preserving sublattice.
Note that Gr\"atzer and Knapp~\cite{gratzerknapp1} called join-distributive
planar semimodular lattices \emph{slim}.
Join-distributive lattices were introduced in many papers  in many different ways and with different names; see Proposition 2.1 in Cz\'edli~\cite{czg:coord} for surveying eight possible definitions, Proposition 6.1 in Cz\'edli and Adaricheva~\cite{czgadari} for two additional definitions, and see Monjardet~\cite{monjardet} and   Adaricheva, Gorbunov, and Tumanov~\cite{adarigorbtum} for further information.

\begin{figure}[h]
\centerline
{\includegraphics[width=\textwidth]{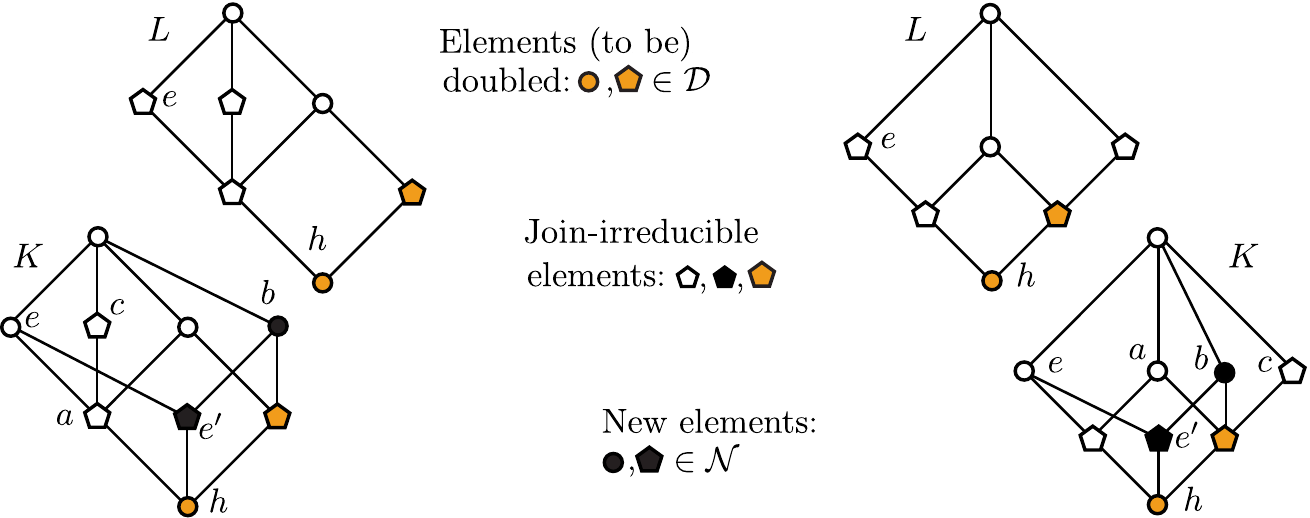}}      %[scale=0.93]{czgamfig1}}
\caption{Neither modularity, nor join-distributivity is preserved}\label{fig2}
\end{figure}

The following statement is well known; we only present it here to show an interesting easy application of our results.

\begin{corollary} The length of a finite distributive lattice $L$ is $|\Jir L|$.
\end{corollary}

\begin{proof} Theorem~\ref{thm:infrect} applied for the trivial partition $\set{\set u:u\in \Jir L}$, Remark~\ref{rem:kthvtkhsnrfdn}, and Lemma~\ref{lemma:dist}  yield an  extension of $L$ into  a boolean lattice $K$ such that  $|\Jir K|=|\Jir L|$ and $\length K=\length L$. Since $K$ is isomorphic to the powerset lattice $(\Pow{\Jir K},\subseteq)$, it is trivial that $\length K=|\At K|=|\Jir K|$.
The equalities listed so far imply the corollary.
\end{proof}

We continue the paper with three lemmas that give a better understanding of the $\Lext$ construction.

\begin{lemma}\label{lemma:cNvSx}
With the assumptions and notations of Definition~\ref{def:Lext} and  Theorem~\ref{thmmain}, $D$ and $D\cup N$ are convex subsets of $K=\Lext$ and $h$ is their common smallest element. Also, $D$ is a convex subset of $L$ and $N$ is a convex subset of $K$.
\end{lemma}

\begin{proof} Since $L$ is a cover-preserving sublattice of $K$, the lemma is just a transcript of \eqref{eq:hmgRsknlnKsLt} by the canonical correspondence. 
\end{proof}

A possible way to understand $\Lext$ is to derive  its Hasse diagram from that of $L$. By describing the covering relation in $\Lext$, the following lemma is useful.

\begin{lemma}\label{lemma:bdscgrpcv}
With the assumptions and notations of Definition~\ref{def:Lext} and Theorem~\ref{thmmain}, we have that for any $x,y\in K=\Lext$,
\begin{equation}x \prec_K y \iff 
\begin{cases}
x,y\in L\text{ and }x\prec_L y\qquad\text{ or}\cr
x,y\in N\text{ and }\alit x\prec_L \alit y\qquad\text{ or}\cr
x\in L,\text{ } y\in N,\text{ and }x=\alit y \qquad\text{ or}\cr
%x\in N,\text{ } y=e,\text{ and }\alit x\in\Max{D\cap\sideal L e} \qquad\text{ or}\cr
x\in N,\text{ } y=\alit x \svee L e,\text{ and }\length{[\alit x, \alit x\svee L e]_L}=2.
\end{cases}
\label{eq:lfltjklvrkmszkhtdbtfjmnpk}
\end{equation} 
\end{lemma}

\begin{proof} 
For $x,y\in L$, \eqref{eq:lfltjklvrkmszkhtdbtfjmnpk} is clear, because $L$ is a cover-preserving sublattice of $K$. 
The second line of \eqref{eq:lfltjklvrkmszkhtdbtfjmnpk} follows from the second line of \eqref{eq:zSznJznKgTsh} since $N$ is a convex subset of $K$ by  Lemma~\ref{lemma:cNvSx}. 
For $x\in L$ and  $y\in N$, the statement is trivial by the third line of \eqref{eq:zSznJznKgTsh}. 

So, in the rest of the  proof, we assume that $x\in N$ and $y\in L$. Note that $x\neq \alit x\svee L e\in L$.
First, assume that $x\prec_K y$. By the last line of \eqref{eq:zSznJznKgTsh}, the least element of $\sfilter K x\cap L$ is $\alit x\svee L e$, whence $y=\alit x\svee L e$. Since $\alit x\prec_K x$ and $x\prec_K y=\alit x\svee L e$, (JHCC) yields that $\length{[\alit x, \alit x\svee L e]_L}=\length{[\alit x, \alit x\svee L e]_K}=2$, as required. 

Second, assume that  $y=\alit x\svee L e$ and $\length{[\alit x, \alit x\svee L e]_L}=2$. Using the last line of \eqref{eq:zSznJznKgTsh} and that $x\neq \alit x\svee L e$, we have that $x<_K\alit x\svee L e=y$. Since 
$\length{[\alit x,y]_K}=\length{[\alit x,y]_L}=2$
and $\alit x\prec_K x<_K y$, (JHCC) implies that  
 $x\prec_K y$, completing the proof. 
\end{proof}

\begin{lemma}\label{lemma:jNbt} With the assumptions and notations of Definition~\ref{def:Lext} and Theorem~\ref{thmmain}, we have that for any $x,y\in K=\Lext$,
\begin{equation}x\svee K y = 
\begin{cases}
x\svee L y,&\text{if }x,y\in L,\cr
\lift{(\alit x\svee L \alit y)},&\text{if }x,y\in N\text{ and }\alit x\svee L \alit y\in D,\cr
\alit x\svee L \alit y\svee L e,&\text{if }x,y\in N\text{ and }\alit x\svee L \alit y\not\in D,\cr
\lift{(\alit x\svee L  y)},&\text{if }x\in N,\text{ }y\notin N,\text{ and }\alit x\svee L  y\in D,\cr
\alit x\svee L y\svee L e,&\text{if }x\in N,\text{ }y\notin N\text{ and }\alit x\svee L  y\not\in D,\cr
\lift{(x\svee L  \alit y)},&\text{if }x\notin N,\text{ }y\in N,\text{ and }x\svee L  \alit y\in D,\cr
x\svee L \alit y\svee L e,&\text{if }x\notin N,\text{ }y\in N\text{ and }x\svee L  \alit y\not\in D.
\end{cases}
\label{eq:sSzgklTkprjhFcbz}
\end{equation} 
\end{lemma}

Of course, ``$\in N$'' and ``$\notin N$'' in \eqref{eq:sSzgklTkprjhFcbz} are equivalent to ``$\notin L$'' and ``$\in L$'', respectively. 

\begin{proof}[Proof of Lemma~\ref{lemma:jNbt}] Since we already know from Theorem~\ref{thmmain} that $L$ is a sublattice of $K$, the first line of \eqref{eq:sSzgklTkprjhFcbz} is clear. The rest of \eqref{eq:sSzgklTkprjhFcbz} follows by straightforward considerations and since $x$ and $y$ play a symmetrical role. Hence, we only give the details for the fifth row of  \eqref{eq:sSzgklTkprjhFcbz}, where $x\in N$,  $y\notin N$, and  $z:=\alit x\svee L y\notin D$. We need to show that $z\svee L e$ is the least upper bound of $\set{x,y}$ in $K$. The inequality  
$x\leq_K z\svee L e$ is clear by the fourth line of \eqref{eq:zSznJznKgTsh} while  $y\leq_K z\svee L e$ follows from $y\leq_L z$.
Hence, $z\svee L e$ is an upper bound of $\set{x,y}$ in $K$. Let $t\in K$ be another upper bound. We claim that $t\notin N$. For the sake of contradiction, suppose that $t\in N$. Then  \eqref{eq:zSznJznKgTsh} yields that $\alit x\leq_L \alit t$ and $y\leq_L \alit t$. Hence, $\alit x\leq_L z= \alit x\svee L y\leq_L \alit t$. Thus,  Lemma~\ref{lemma:cNvSx} gives that $z\in D$, which is a contradiction showing that $t\notin N$. 
Now $x\leq_K t$ and \eqref{eq:zSznJznKgTsh} gives that 
$\alit x\svee L e\leq_L t$. Also,  $y\leq_K t$ gives $y\leq_L t$.   Hence, 
$ z\svee L e =\alit x\svee L y\svee L e  \leq_L t$,
whereby  $x\svee K y=z\svee L e$, as required.
\end{proof}

Finally, we present a proof promised in Section~\ref{sect:visualintro}.

\begin{proof}[Proof of Observation~\ref{observe:nlkrttHgb}] For the sake of contradiction, suppose that $K$ 
is a length-preserving semimodular extension  of $L$ such that $\Jir K\cong P'$. (We do not assume that $a\in L$ corresponds to $a'\in K$, etc.)
Since $e'\not\geq c'$ excludes that $e'=1_K$ and $\length K=\length L=4$, it follows that $0\prec_K b'\prec_K a'\prec_K e'\prec_K 1_K$. Using that $d'\not\leq_K a'\leq_K e'$, $d'\leq_K e'$, so  $a'<_K a'\vee d'\leq e'$, and $a'\prec_K e'$, and  we obtain that $ a'\vee d' =  e'$, contradicting $e'\in\Jir K$.  
\end{proof}

\renewcommand{\refname}{References}

\makeatletter
\renewcommand\@biblabel[1]{[#1]}
\makeatother

\color{black}

%{
%%Font for the figures
%$habcdpqree'H \alg N \alg M$ $\lcov L e=a$ $\lcov K{e'}=h$
%
%$\{ $ $a,b,cdPP'LKz_0,z_1,z_2,B_0, B, A,\sideal P e,\sodeal P e$,
%$\sideal Q e,\sodeal Q e,L, K, e'$ \\
%Elements (to be) doubled:   $\in \alg D$  New elements:   $\in \alg N$
%Join-irreducible elements: 
%
%$L=L_0$ $K_1=L_1$  
%$L_1=K_1$ $K_2=L_2$  
%$L_2=K_2$ $K_3=K$  
%
%
%}

\end{document}